\documentclass[10pt,leqno]{article}

\usepackage{amsmath,amssymb,amsthm,mathrsfs,dsfont}

\usepackage[margin=3cm]{geometry} 

\usepackage{titlesec,hyperref}

\usepackage{color}

\usepackage{fancyhdr}
\titleformat{\subsection}{\it}{\thesubsection.\enspace}{1pt}{}

\newtheorem{theo}{Theorem}[section]
\newtheorem{lemm}[theo]{Lemma}

\newtheorem{prop}[theo]{Proposition}
\newtheorem{rema}[theo]{Remark}
\numberwithin{equation}{section}

\allowdisplaybreaks 

\begin{document}
\title{Vanishing viscosity limit to the FENE dumbbell model of polymeric flows
\hspace{-4mm}
}

\author{ Zhaonan $\mbox{Luo}^1$ \footnote{email: 1411919168@qq.com},\quad
Wei $\mbox{Luo}^1$\footnote{E-mail:  luowei23@mail2.sysu.edu.cn} \quad and\quad
 Zhaoyang $\mbox{Yin}^{1,2}$\footnote{E-mail: mcsyzy@mail.sysu.edu.cn}\\
 $^1\mbox{Department}$ of Mathematics,
Sun Yat-sen University, Guangzhou 510275, China\\
$^2\mbox{Faculty}$ of Information Technology,\\ Macau University of Science and Technology, Macau, China}

\date{}
\maketitle
\hrule

\begin{abstract}
In this paper we mainly investigate the inviscid limit for the strong solutions of the finite extensible nonlinear elastic (FENE) dumbbell model. By virtue of the Littlewood-Paley theory, we first obtain a uniform estimate for the solution to the FENE dumbbell model with viscosity in Besov spaces. Moreover, we show that the data-to-solution map is continuous. Finally, we prove that the strong solution of the FENE dumbbell model converges to a Euler system couple with a Fokker-Planck equation. Furthermore, convergence rates in Lebesgue spaces are obtained also. \\

\vspace*{5pt}

\vspace*{5pt}
\end{abstract}

\vspace*{10pt}

\tableofcontents

\section{Introduction}
\par
In this paper we investigate the vanishing viscosity limit problem to the finite extensible nonlinear elastic (FENE) dumbbell model \cite{Bird1977,Doi1988}:
\begin{align}\label{1.1}
\left\{
\begin{array}{ll}
u_t+(u\cdot\nabla)u-\frac {\epsilon} {Re} \Delta{u}+\nabla{P}=\frac {1-\epsilon} {ReWe}div~\tau, ~~~~~~~div~u=0,\\[1ex]
\psi_t+(u\cdot\nabla)\psi=div_{R}[- \sigma(u)\cdot{R}\psi+\frac {1} {2WeN}\nabla_{R}\psi+\frac {1} {2We}\nabla_{R}\mathcal{U}\psi],  \\[1ex]
\tau_{ij}=\epsilon \int_{B}(R_{i}\nabla_{j}\mathcal{U})\psi dR, \\[1ex]
u|_{t=0}=u_0,~~\psi|_{t=0}=\psi_0, \\[1ex]
(\beta\nabla_{R}\psi+\nabla_{R}\mathcal{U}\psi)\cdot{n}=0 ~~~~ \text{on} ~~~~ \partial B(0,R_{0}) .\\[1ex]
\end{array}
\right.
\end{align}
In \eqref{1.1}~~$u(t,x)$ stands for the velocity of the polymeric liquid. The polymer particles are described by the distribution function $\psi(t,x,R)$. Here the polymer elongation $R$ is bounded in ball $ B=B(0,R_{0})$ which means that the extensibility of the polymers is finite and $x\in\mathbb{R}^n$.
$\tau$ is an extra-stress tensor which generated by the polymer particles effect and $P$ is the pressure. The Reynolds number $Re=\frac{\epsilon}{\nu}$ with $\nu$ is the viscosity of the fluid, the viscosity ratio between the polymer and the solvent $\epsilon\in(0,1)$ and the density $\rho=\int_B\psi dR$. Moreover the potential $\mathcal{U}(R)=-k\log(1-(\frac{|R|}{|R_{0}|})^{2})$ for some $k>0$. $\sigma(u)$ is the drag term. In general, $\sigma(u)=\nabla u$. For the co-rotation case, $\sigma(u)=\frac{\nabla u-(\nabla u)^{T}}{2}$.

For simplicity, we assume that $\beta=1$, $We=1$, $N=1$, $R_{0}=1$ and $\nu=\frac {\epsilon} {Re}\in[0,1]$.
Notice that $(u,\psi)=(0,\psi_{\infty})$ with $$\psi_{\infty}(R)=\frac{e^{-\mathcal{U}(R)}}{\int_{B}e^{-\mathcal{U}(R)}dR}=\frac{(1-|R|^2)^k}{\int_{B}(1-|R|^2)^kdR},$$
is a trivial solution of \eqref{1.1}.

By a simple calculation, we can rewrite \eqref{1.1} for the following system:
\begin{align}\label{eq1}
\left\{
\begin{array}{ll}
u_t+(u\cdot\nabla)u-\frac {\epsilon} {Re}\Delta u+\nabla{P}=\frac {1-\epsilon} {Re}div~\tau,  ~~~~~~~div~u=0,\\[1ex]
\psi_t+(u\cdot\nabla)\psi=div_{R}[-\sigma(u)\cdot{R}\psi+\frac {1} {2}\psi_{\infty}\nabla_{R}\frac{\psi}{\psi_{\infty}}],  \\[1ex]
\tau_{ij}=\epsilon\int_{B}(R_{i}\nabla_{R_j}\mathcal{U})\psi dR, \\[1ex]
u|_{t=0}=u_0,~~ \psi|_{t=0}=\psi_0, \\[1ex]
\psi_{\infty}\nabla_{R}\frac{\psi}{\psi_{\infty}}\cdot{n}=0 ~~~~ \text{on} ~~~~ \partial B(0,1) .\\[1ex]
\end{array}
\right.
\end{align}

{\bf Remark.} As in the reference \cite{Masmoudi2013}, one can deduce that $\psi=0$ on $\partial B(0,1)$.

There are a lot of mathematical results about the dumbbell model. M. Renardy \cite{Renardy} established the local well-posedness in Sobolev spaces with potential $\mathcal{U}(R)=(1-|R|^2)^{1-\sigma}$ for $\sigma>1$. Later, B. Jourdain, T. Leli\`{e}vre, and
C. Le Bris \cite{Jourdain} proved local existence of a stochastic differential equation with potential $\mathcal{U}(R)=-k\log(1-|R|^{2})$ in the case $k>3$ for a Couette flow. H. Zhang and P. Zhang \cite{Zhang-H} proved local well-posedness of \eqref{eq1} with $d=3$ in weighted Sobolev spaces. For the co-rotation case, F. Lin, P. Zhang, and Z. Zhang \cite{F.Lin} obtain a global existence results with $d=2$ and $k > 6$. If the initial data is perturbation around equilibrium, N. Masmoudi \cite{Masmoudi2008} proved global well-posedness of \eqref{eq1} for $k>0$. In the co-rotation case with $d=2$, he \cite{Masmoudi2008} obtained a global result for $k>0$ without any small conditions. W. Luo and Z. Yin improved the result to Besov spaces \cite{Luo-Yin-NA}. In the co-rotation case, A. V. Busuioc, I. S. Ciuperca, D. Iftimie and L. I. Palade \cite{Busuioc} obtain a global existence result with only the small condition on $\psi_0$. The global existence of weak solutions in $L^2$ was proved recently by N. Masmoudi \cite{Masmoudi2013} under some entropy conditions.

Recently, M. Schonbek \cite{Schonbek} studied the $L^2$ decay of the velocity for the co-rotation
FENE dumbbell model, and obtained the
decay rate $(1+t)^{-\frac{d}{4}+\frac{1}{2}}$ with $d\geq 2$ and $u_0\in L^1$.
Moreover, she conjectured that the sharp decay rate should be $(1+t)^{-\frac{d}{4}}$,~$d\geq 2$.
However, she failed to get it because she could not use the bootstrap argument as in \cite{Schonbek1985} due to the
additional stress tensor.  More recently, W. Luo and Z. Yin \cite{Luo-Yin,Luo-Yin2} improved the decay rate to $(1+t)^{-\frac{d}{4}}$ with $d\geq 2$.

The vanishing viscosity limit problem of the incompressible Navier-Stokes equation has been studied widely. The first result was obtained by H. Swann\cite{Swann} with $u_0\in H^s$ for $s>3$. T. Kato \cite{Kato} proved Swann's result by a simple method. They prove the strong solutions are convergence in a weaker topology $H^{s'}$ with $s'<s$. A. Majda\cite{Majda} prove the convergence in $L^2$ space with $u_0\in L^2$.  N. Masmoudi \cite{Masmoudi-2007CMP} improve the convergence in $H^s$ with $u_0\in H^s$. Recently, Z. Guo, J. Li and Z. Yin \cite{Guo-Li-Yin} generalized the previous result in Besov spaces. F. Bernicot, T. Elgindi and S. Keraani \cite{Bernicot} prove the global convergence with vorticity belonging to BMO-type spaces in the case $d=2$.

To our best knowledge, there is no result about the vanishing viscosity limit to \eqref{eq1}.  In this paper, we investigate this problem.
Firstly, we investigate the local well-posedness of \eqref{eq1} in $B^s_{p,r}\times B^{s-1}_{p,r}(\mathcal{L}^p)$ by using approximate argument. The key point is to prove some a priori estimates for the linear problem of \eqref{eq1} by transport-diffusion equations theory and the Littlewood-Paley theory. The most difficult term is the additional stress tensor $\tau$ with low regularity. According to transport-diffusion equations theory, we must estimate $\tau$ in $\tilde{L}^{\infty}_{T}(B^{s-1}_{p,r})$ to get uniform bound for velocity $u$ in $\tilde{L}^{\infty}_{T}(B^{s}_{p,r})$ with respect to $\nu$. For this purpose, we have to estimate $\psi$ or $\nabla_{R} \psi$ in $\tilde{L}^{\infty}_{T}(B^{s-1}_{p,r}(\mathcal{L}^p))$. However, we can't get any information of $\nabla_{R} \psi~in~L^{\infty}_{T}$ from the Fokker-Planck equation. Thus we have to estimate $\psi$ and add an additional condition $k(p-1)>1$ to prove $\tau~\in~\tilde{L}^{\infty}_{T}(B^{s-1}_{p,r})$ by using the useful lemma in \cite{Masmoudi2008}.
Next, we study uniform continuous dependence for $\nu$ and the inviscid limit in the same topology $B^s_{p,r}\times B^{s-1}_{p,r}(\mathcal{L}^p)$.
The proof is based on the Bona-Smith method but in the $L^p$ setting. The main idea is to smooth out the initial data by the frequency localization operator $S_N$ defined in Section 2 and to apply the uniform bound for the higher regularity solution of \eqref{eq1} with the smooth initial data.
Finally, we prove the rates of convergence in $L^p$ by transport equations theory and interpolation theory. In order to get convergence rates, we estimate the convergence of velocity
in $B^{s-2}_{p,r}$ since we have uniform bound for $\Delta u$ and $div~\tau$ in this space. Using interpolation theory, we discovery that convergence rates in $L^p$ depend on regularity.

The paper is organized as follows. In Section 2, we introduce some notations and  our main results. In Section 3, we state some lemmas which will be useful in this paper. In Section 4, we consider the linear problem of \eqref{eq1} and prove some priori estimates for solutions to \eqref{eq1}. In Section 5, we prove the local well-posedness of \eqref{eq1} in $B^s_{p,r}\times B^{s-1}_{p,r}(\mathcal{L}^p)$ by using approximate argument. In Section 6, we present uniform continuous dependence for $\nu$ and then prove
that the strong solution of the FENE dumbbell model converges to a Euler system couple with a
Fokker-Planck equation. In Section 7, we prove the rates of convergence in $L^p$ by transport equations theory and interpolation theory.

\section{Notations and main results}
\par
In this section we will introduce main results and some notations in \cite{Bahouri2011,Luo-Yin-NA} that we shall use throughout the paper.

For $p\geq1$, we denote by $\mathcal{L}^{p}$ the space
$$\mathcal{L}^{p}=\big\{\psi \big|\|\psi\|^{p}_{\mathcal{L}^{p}}=\int_{B} \psi_{\infty}|\frac{\psi}{\psi_{\infty}}|^{p}dR<\infty\big\}.$$

  We will use the notation $L^{p}_{x}(\mathcal{L}^{q})$ to denote $L^{p}(\mathbb{R}^{d};\mathcal{L}^{q}):$
$$L^{p}_{x}(\mathcal{L}^{q})=\big\{\psi \big|\|\psi\|_{L^{p}_{x}(\mathcal{L}^{q})}=(\int_{\mathbb{R}^{d}}(\int_{B} \psi_{\infty}|\frac{\psi}{\psi_{\infty}}|^{q}dR)^{\frac{p}{q}}dx)^{\frac{1}{p}}<\infty\big\}.$$

Let $\mathcal{C}$ be the annulus $\{\xi\in\mathbb{R}^d:\frac 3 4\leq|\xi|\leq\frac 8 3\}$. There exist radial functions $\chi$ and $\varphi$, valued in the interval $[0,1]$, belonging respectively to $\mathcal{D}(B(0,\frac 4 3))$ and $\mathcal{D}(\mathcal{C})$, and such that
$$ \forall\xi\in\mathbb{R}^d,\ \chi(\xi)+\sum_{j\geq 0}\varphi(2^{-j}\xi)=1, $$
$$ \forall\xi\in\mathbb{R}^d\backslash\{0\},\ \sum_{j\in\mathbb{Z}}\varphi(2^{-j}\xi)=1, $$
$$ |j-j'|\geq 2\Rightarrow\mathrm{Supp}\ \varphi(2^{-j}\cdot)\cap \mathrm{Supp}\ \varphi(2^{-j'}\cdot)=\emptyset, $$
$$ j\geq 1\Rightarrow\mathrm{Supp}\ \chi(\cdot)\cap \mathrm{Supp}\ \varphi(2^{-j}\cdot)=\emptyset. $$
The set $\widetilde{\mathcal{C}}=B(0,\frac 2 3)+\mathcal{C}$ is an annulus, and we have
$$ |j-j'|\geq 5\Rightarrow 2^{j}\mathcal{C}\cap 2^{j'}\widetilde{\mathcal{C}}=\emptyset. $$
Further, we have
$$ \forall\xi\in\mathbb{R}^d,\ \frac 1 2\leq\chi^2(\xi)+\sum_{j\geq 0}\varphi^2(2^{-j}\xi)\leq 1, $$
$$ \forall\xi\in\mathbb{R}^d\backslash\{0\},\ \frac 1 2\leq\sum_{j\in\mathbb{Z}}\varphi^2(2^{-j}\xi)\leq 1. $$

$\mathcal{F}$ represents the Fourier transform and  its inverse is denoted by $\mathcal{F}^{-1}$.
Let $u$ be a tempered distribution in $\mathcal{S}'(\mathbb{R}^d)$. For all $j\in\mathbb{Z}$, define
$$
\Delta_j u=0\,\ \text{if}\,\ j\leq -2,\quad
\Delta_{-1} u=\mathcal{F}^{-1}(\chi\mathcal{F}u),\quad
\Delta_j u=\mathcal{F}^{-1}(\varphi(2^{-j}\cdot)\mathcal{F}u)\,\ \text{if}\,\ j\geq 0,\quad
S_j u=\sum_{j'<j}\Delta_{j'}u.
$$
Then the Littlewood-Paley decomposition is given as follows:
$$ u=\sum_{j\in\mathbb{Z}}\Delta_j u \quad \text{in}\ \mathcal{S}'(\mathbb{R}^d). $$

Let $s\in\mathbb{R},\ 1\leq p,r\leq\infty.$ The nonhomogeneous Besov space $B^s_{p,r}$ and $B^s_{p,r}(\mathcal{L}^q)$ is defined by
$$ B^s_{p,r}=\{u\in S':\|u\|_{B^s_{p,r}}=\Big\|(2^{js}\|\Delta_j u\|_{L^p})_j \Big\|_{l^r(\mathbb{Z})}<\infty\}, $$
$$ B^s_{p,r}(\mathcal{L}^q)=\{\phi\in S':\|\phi\|_{B^s_{p,r}(\mathcal{L}^q)}=\Big\|(2^{js}\|\Delta_j \phi\|_{L_{x}^{p}(\mathcal{L}^q)})_j \Big\|_{l^r(\mathbb{Z})}<\infty\}.$$
We denote $C_T(B^s_{p,r})$ by $C([0,T];B^s_{p,r})$ and $L^{\rho}_{T}(B^s_{p,r})$ by $L^{\rho}([0,T];B^s_{p,r})$ respectively. Moreover we will use the spaces
$\tilde{L}^{\rho}_{T}(B^s_{p,r})$ and $\tilde{L}^{\rho}_{T}(B^s_{p,r}(\mathcal{L}^q))$
$$ \tilde{L}^{\rho}_{T}(B^s_{p,r})=\{u\in S':\|u\|_{\tilde{L}^{\rho}_{T}(B^s_{p,r})}=\Big\|(2^{js}\|\Delta_j u\|_{L_{T}^{\rho}(L^p)})_j \Big\|_{l^r(\mathbb{Z})}<\infty\}, $$
$$ \tilde{L}^{\rho}_{T}(B^s_{p,r}(\mathcal{L}^q))=\{\phi\in S':\|\phi\|_{\tilde{L}^{\rho}_{T}(B^s_{p,r}(\mathcal{L}^q))}=\Big\|(2^{js}\|\Delta_j \phi\|_{L_{T}^{\rho}(L_{x}^{p}(\mathcal{L}^q))})_j \Big\|_{l^r(\mathbb{Z})}<\infty\}.$$
The fact $L^{1}_{T}(B^s_{p,r})\hookrightarrow\tilde{L}^{1}_{T}(B^s_{p,r})$, ~ $L^{1}_{T}(B^s_{p,r}(\mathcal{L}^q))\hookrightarrow\tilde{L}^{1}_{T}(B^s_{p,r}(\mathcal{L}^q))$, ~$\tilde{L}^{\infty}_{T}(B^s_{p,r})\hookrightarrow L^{\infty}_{T}(B^s_{p,r})$ and $\tilde{L}^{\infty}_{T}(B^s_{p,r}(\mathcal{L}^q))\hookrightarrow L^{\infty}_{T}(B^s_{p,r}(\mathcal{L}^q))$ will be useful in this paper.  \\
Next we define a special space $\tilde{E}^s_{p,r}$:
$$ \tilde{E}^s_{p,r}(T)=\{\psi:\|\psi\|_{\tilde{E}^s_{p,r}(T)}=\Big\|\Big(2^{js}(\int_{0}^{T}\int_{\mathbb{R}^{d}\times B} \psi_{\infty}|\nabla_{R}(\frac{\Delta_j\psi}{\psi_{\infty}})^{\frac{p}{2}}|^{2}dRdxdt)^{\frac{1}{p}}\Big)_j \Big\|_{l^r(\mathbb{Z})}<\infty\}.$$

The main results of this paper can be stated as follows.

\begin{theo}\label{th1}
Let $d\geq 2$. Assume that $\nu=\frac {\epsilon} {Re} \in [0,1]$, $s>1+max\{\frac 1 2, \frac d p\}$, $2\leq p\leq r<\infty$ and $k(p-1)>1$. Then for any $K>0$, $(u_0,\psi_0)\in B_K=\{(v,\phi)\in B^s_{p,r}\times B^{s-1}_{p,r}(\mathcal{L}^p):\|v\|_{B^s_{p,r}}+\|\phi\|_{B^{s-1}_{p,r}(\mathcal{L}^p)}\leq K, div~v=0\}$, there exists $T=T(K,s,p,r,d)>0$ such that the equation \eqref{eq1} has a unique solution $H^{\nu}(u_0,\psi_0)=(u^{\nu},\psi^{\nu})\in C_T(B^s_{p,r})\times C_T( B^{s-1}_{p,r}(\mathcal{L}^p))$. Moreover, $\psi^{\nu}\in \tilde{E}^{s-1}_{p,r}(T)$ and we have   \\

(1) (Uniform bounds): there exists $C=C(K,s,p,r,d)>0$ such that
\begin{align}
\|H^{\nu}(u_0,\psi_0)\|_{\tilde{L}_T^{\infty}(B^s_{p,r})\times\tilde{L}_T^{\infty}(B^{s-1}_{p,r}(\mathcal{L}^p))}\leq C, \quad \forall\nu\in[0,1].
\end{align}
Moreover, if $(u_0,\psi_0)\in B^{\gamma}_{p,r}\times B^{\gamma-1}_{p,r}(\mathcal{L}^p)$ for any $\gamma>s$, then there exists $C_1=C_1(K,\gamma,s,p,r,d)>0$ such that
\begin{align}
\|H^{\nu}(u_0,\psi_0)\|_{\tilde{L}_T^{\infty}(B^{\gamma}_{p,r})\times\tilde{L}_T^{\infty}(B^{\gamma-1}_{p,r}(\mathcal{L}^p))}\leq C_1(\|u_{0}\|_{B^{\gamma}_{p,r}}+\|\psi_{0}\|_{B^{\gamma-1}_{p,r}(\mathcal{L}^p)}), \quad \forall\nu\in[0,1].
\end{align}

(2) (Uniform continuous dependence for $\nu$): the solution map $(u_0,\psi_0)\rightarrow (u^{\nu},\psi^{\nu})$ is continuous from $B_K$ to $C_T(B^s_{p,r})\times C_T( B^{s-1}_{p,r}(\mathcal{L}^p))$ uniformly with respect to $\nu$. Namely, for any $\eta>0$, there exists a $\delta=\delta(u_{0},\psi_{0},K,s,p,r,d)>0$ such that for any $(v_0,\phi_0)\in B_K$ with $\|v_0-u_{0}\|_{B^{s}_{p,r}}+\|\phi_0-\psi_{0}\|_{B^{s-1}_{p,r}(\mathcal{L}^p)}<\delta$, we have
\begin{align}
\|H^{\nu}(u_0,\psi_0)-H^{\nu}(v_0,\phi_0)\|_{\tilde{L}_T^{\infty}(B^{s}_{p,r})\times\tilde{L}_T^{\infty}(B^{s-1}_{p,r}(\mathcal{L}^p))}<\eta, \quad \forall\nu\in[0,1].
\end{align}

(3) (Inviscid limit): we deduce that
\begin{align}
\lim_{\nu\rightarrow 0}
\|H^{\nu}(u_0,\psi_0)-H^{0}(u_0,\psi_0)\|_{\tilde{L}_T^{\infty}(B^{s}_{p,r})\times\tilde{L}_T^{\infty}(B^{s-1}_{p,r}(\mathcal{L}^p))}=0.
\end{align}

\end{theo}

\begin{rema}\label{remark}
Since we can't get any information of $\nabla_{R} \psi~in~L^{\infty}_{T}$ from the Fokker-Planck equation, we need not to prove $\psi\in \tilde{E}^{s-1}_{p,r}(T)$. Thus, we can prove Theorem \ref{th1} by replacing the assumption $2\leq p\leq r<\infty$ with the assumption $2\leq p,~r<\infty$.
\end{rema}

\begin{theo}\label{th2}
Under the assumption of Theorem \ref{th1}, we deduce that the rates of convergence in $L^p$ for $t\in[0,T]$: if $s>2$, then
$$\|u^\nu-u^0\|_{L^p}\leq C\nu,~~\|\psi^\nu-\psi^0\|_{L_x^p(\mathcal{L}^p)}\leq C\nu^{\frac 1 2}.$$
If $s=2$, for any small $\epsilon_1\in(0,1)$, we have
$$\|u^\nu-u^0\|_{L^p}\leq C\nu^{1-\epsilon_1},~~\|\psi^\nu-\psi^0\|_{L_x^p(\mathcal{L}^p)}\leq C\nu^{\frac {1-\epsilon_1} {2}}.$$
If $s<2$, we get
$$\|u^\nu-u^0\|_{L^p}\leq C\nu^{\frac s 2},~~\|\psi^\nu-\psi^0\|_{L_x^p(\mathcal{L}^p)}\leq C\nu^{\frac {s-1} {2}}.$$
\end{theo}

\begin{rema}\label{remark}
Under the assumption of Theorem \ref{th1}, we assume further that $s>2+max\{1, \frac d p\}$. Then we can prove the optimal convergence rates:
$\|u^\nu-u^0\|_{L^p}\leq C\nu,~~\|\psi^\nu-\psi^0\|_{L_x^p(\mathcal{L}^p)}\leq C\nu$ by estimating $\psi^\nu-\psi^0$ in $L_{T}^{\infty}(B^{s-3}_{p,r}(\mathcal{L}^p))$.
\end{rema}

\begin{rema}
Note that $B^s_{2,2}=H^s$. Our results cover the case in Sobolev spaces $H^s$  by taking $p=r=2$ in Theorem \ref{th1}. Our results on convergence rates cover the case for the incompressible Navier-Stokes equation in \cite{Bahouri2011,Masmoudi-2007CMP}  by taking $p=r=2$ and $\psi=0$ in Theorem \ref{th2}.
\end{rema}

\section{Preliminaries}
\par

First, we recall nonhomogeneous Bony's decomposition from \cite{Bahouri2011}.
$$uv=T_{u}v+T_{v}u+R(u,v),$$
with
$$T_{u}v=\sum_{j}S_{j-1}u\Delta_{j}v,\quad R(u,v)=\sum_{j}\sum_{|k-j|\leq1}\Delta_{j}u\Delta_{k}v.$$

This is a standard tool for nonlinear estimates. We can use Bony's decomposition to prove product laws:
\begin{lemm}\label{product}\cite{Bahouri2011}
(1) For any $s>0$ and any $(p,r)$ in $[1,\infty]^2$, the space $L^{\infty} \cap B^s_{p,r}$ is an algebra, and a constant $C=C(s,d)$ exists such that
$$ \|uv\|_{B^s_{p,r}}\leq C(\|u\|_{L^{\infty}}\|v\|_{B^s_{p,r}}+\|u\|_{B^s_{p,r}}\|v\|_{L^{\infty}}). $$
(2) If $2\leq p\leq \infty,\ 1\leq r\leq \infty,\ s_1\leq s_2,\ s_2>\frac{d}{p} (s_2 \geq \frac{d}{p}\ \text{if}\ r=1)$ and $s_1+s_2>0$, there exists $C=C(s_1,s_2,p,r,d)$ such that
$$ \|uv\|_{B^{s_1}_{p,r}}\leq C\|u\|_{B^{s_2}_{p,r}}\|v\|_{B^{s_1}_{p,r}}$$
and
$$ \|u\psi\|_{B^{s_1}_{p,r}(\mathcal{L}^p)}\leq C\|u\|_{B^{s_2}_{p,r}}\|\psi\|_{B^{s_1}_{p,r}(\mathcal{L}^p)}. $$
\end{lemm}

The following lemma will be useful to deal with the estimate of pressure term $\nabla P$.
\begin{lemm}\label{P}
For $(p,r)$ in $[1,\infty]^2$, $s>1+\frac d p$ and $s-2\leq\sigma\leq s$, there exists a constant $C=C(s,d,p,r)$ such that if $div~u=div~v=0$, then
\begin{align}\label{P1}
\|\nabla(-\Delta)^{-1}div(u\cdot\nabla v)\|_{B^\sigma_{p,r}}\leq C\min(\|u\|_{B^\sigma_{p,r}}\|v\|_{B^s_{p,r}},\|u\|_{B^s_{p,r}}\|v\|_{B^\sigma_{p,r}}).
\end{align}
\end{lemm}
\begin{proof}
By a standard argument as in Lemmas $7.9-7.10$ in \cite{Bahouri2011}, using Bony's decomposition and the fact that $B^s_{p,r}\hookrightarrow C^{0,1}$, we can check that \eqref{P1} is true. We omit the details here.
\end{proof}

The following lemma allows us to estimate the extra stress tensor $\tau$.
\begin{lemm}\label{Lemma3}
\cite{Masmoudi2008} Let $z=1-|R|$. If $(p-1)k>1$, then
$$(\int_{B}\frac{|\psi|}{z} dR)^p\leq C\int_{B}|\frac{\psi}{\psi_{\infty}}|^{p}\psi_{\infty}dR.$$
\end{lemm}

\begin{lemm}\label{new estimate}\cite{Bahouri2011,Li-Yin,Luo-Yin-NA}
Let $1\leq p\leq\infty,\ 1\leq r\leq\infty,\ \sigma>-1-d\min(\frac 1 {p}, \frac 1 {p'})$ and $div~u=0$. Define $R_j=[u\cdot\nabla, \Delta_j]\psi(t,x,R)$. There exists a constant $C$ such that for $\sigma>0$, we have
$$\Big\|(2^{j\sigma}\|R_j\|_{L^p_{x}(\mathcal{L}^p)})_j\Big\|_{l^r(\mathbb{Z})}\leq C(\|\nabla u\|_{L^{\infty}}\|\psi\|_{B^{\sigma}_{p,r}(\mathcal{L}^p)}+\|\nabla u\|_{B^{\sigma-1}_{p,r}}\|\nabla_{x}\psi\|_{L^{\infty}_{x}(\mathcal{L}^p)}).$$
If $\sigma>1+\frac d p$ or $\sigma=1+\frac d p,~r=1$, we have
$$\Big\|(2^{j\sigma}\|R_j\|_{L^p_{x}(\mathcal{L}^p)})_j\Big\|_{l^r(\mathbb{Z})}\leq C\|\nabla u\|_{B^{\sigma-1}_{p,r}}\|\psi\|_{B^\sigma_{p,r}(\mathcal{L}^p)}.$$
If $\sigma=1+\frac d p,~r>1$, we have
$$\Big\|(2^{j\sigma}\|R_j\|_{L^p_{x}(\mathcal{L}^p)})_j\Big\|_{l^r(\mathbb{Z})}\leq C\|\nabla u\|_{B^{\sigma}_{p,r}}\|\psi\|_{B^\sigma_{p,r}(\mathcal{L}^p)}.$$
If $\sigma<1+\frac d p$, we have
$$\Big\|(2^{j\sigma}\|R_j\|_{L^p_{x}(\mathcal{L}^p)})_j\Big\|_{l^r(\mathbb{Z})}\leq C\|\nabla u\|_{B^{\frac {d} {p}}_{p,r}\cap L^{\infty}}\|\psi\|_{B^\sigma_{p,r}(\mathcal{L}^p)}.$$
\end{lemm}

Now we state a useful estimate in the study of transport-diffusion equation, which is crucial to the proofs of our main theorem later.
\begin{equation}\label{tran-diff}
\left\{\begin{array}{l}
    f_t+v\cdot\nabla f-\nu\Delta f=g,\ x\in\mathbb{R}^d,\ t>0, \\
    f(0,x)=f_0(x).
\end{array}\right.
\end{equation}

\begin{lemm}\label{u}\cite{Bahouri2011,Li-Yin}
Let $1\leq p,r\leq\infty,\ \sigma> -d\min(\frac 1 {p}, \frac 1 {p'})$ or $\sigma>-1-d\min(\frac 1 {p}, \frac 1 {p'})$ if $div~v=0$. There exists a constant C ,depending only on $d,p,r,\sigma$, for any smooth solution $f$ of \eqref{tran-diff} and $t\geq0$, such that if $\nu>0,~1\leq \rho\leq\infty$, then
$$\|f(t)\|_{\tilde{L}^{\infty}_{T}({B^\sigma_{p,r}})}\leq Ce^{CV_p(v,t)}\big(\|f_0\|_{B^\sigma_{p,r}}+(1+\nu T)^{1-\frac 1 \rho}\nu^{\frac 1 \rho -1}
\|g\|_{\tilde{L}^{\rho}_{T}(B^{\sigma-2+\frac 2 \rho}_{p,r})}\big),$$
and if $\nu=0$, then
$$\|f(t)\|_{\tilde{L}^{\infty}_{T}({B^\sigma_{p,r}})}\leq Ce^{CV_p(v,t)}\big(\|f_0\|_{B^\sigma_{p,r}}+
\|g\|_{\tilde{L}^{1}_{T}(B^{\sigma}_{p,r})}\big),$$
with
$$
   V'_p(v,t)=\left\{\begin{array}{ll}
    \|\nabla v\|_{B^{\frac d {p}}_{p,\infty}\bigcap L^{\infty}}, &\ \text{if}\ \sigma<1+\frac d {p}, \\
    \|\nabla v\|_{B^{1+\frac d {p} }_{p,r}}, &\ \text{if}\ \sigma=1+\frac d {p}, \ r>1,   \\
    \|\nabla v\|_{B^{\sigma-1}_{p,r}}, &\ \text{if}\ \sigma>1+\frac d {p}\ or\ (\sigma=1+\frac d {p}\ and\ r=1).
  \end{array}\right.
$$
\end{lemm}

\section{Linear problem and a priori estimates}
\par
In this section we will consider the following linear equations for \eqref{eq1}:
\begin{align}\label{eq2}
\left\{
\begin{array}{ll}
u_t+(v\cdot\nabla)u-\frac {\epsilon} {Re}\Delta u+\nabla{P}=\frac {1-\epsilon} {Re}div~\tau+f,  ~~~~~~~div~v=0,\\[1ex]
\psi_t+(v\cdot\nabla)\psi=div_{R}[-\sigma(v)\cdot{R}\psi+\frac {1} {2}\psi_{\infty}\nabla_{R}\frac{\psi}{\psi_{\infty}}+g]+f,  ~~~~~~~div~v=0,\\[1ex]
\tau_{ij}=\epsilon\int_{B}(R_{i}\nabla_{R_j}\mathcal{U})\psi dR, \\[1ex]
u|_{t=0}=u_0,~~ \psi|_{t=0}=\psi_0, \\[1ex]
\psi_{\infty}\nabla_{R}\frac{\psi}{\psi_{\infty}}\cdot{n}=0 ~~~~ \text{on} ~~~~ \partial B(0,1) .\\[1ex]
\end{array}
\right.
\end{align}

\begin{lemm}\label{weak solution}\cite{Luo-Yin-NA}
If $A(t)\in C([0,T])$ is a matrix-valued function and $\psi_0\in \mathcal{L}^p$ with $p\in [2,\infty)$, then
\begin{align}\label{eq3}
\left\{
\begin{array}{ll}
\psi_t=div_{R}[-A(t)\cdot{R}\psi+\frac {1} {2}\psi_{\infty}\nabla_{R}\frac{\psi}{\psi_{\infty}}],  ~~~~~~~\\[1ex]
\psi|_{t=0}=\psi_0, \\[1ex]
\psi_{\infty}\nabla_{R}\frac{\psi}{\psi_{\infty}}\cdot{n}=0 ~~~~ \text{on} ~~~~ \partial B(0,1) \\[1ex]
\end{array}
\right.
\end{align}
has a unique weak solution $\psi\in C([0,T];\mathcal{L}^p)$. Moreover, we obtain
\begin{align*}
\sup_{t\in [0,T]}\int_{B}|\frac {\psi} {\psi_\infty}|^{p}\psi_{\infty}dR+\frac {p-1} {p} \int_{0}^{T}\int_{B}|\nabla_R(\frac {\psi} {\psi_\infty})^{\frac p 2}|^{2}\psi_{\infty}dRdt\leq Ce^{C\int_{0}^{T}A^2(t)dt}\int_{B}|\frac {\psi_0} {\psi_\infty}|^{p}\psi_{\infty}dR.
\end{align*}
\end{lemm}

First, we prove a priori estimate for the Fokker-Planck equation.
\begin{lemm}\label{estimate1}
Let $s>1+max(\frac d p,\frac 1 2)$, $s-2\leq\sigma\leq s-1$, $2\leq p\leq r<\infty$. Assume that $\psi_0\in B^{\sigma}_{p,r}(\mathcal{L}^p)$, $f,g\in L^2_{T}(B^{\sigma}_{p,r}(\mathcal{L}^p))$ and $v\in L^{\infty}_T(B^{s}_{p,r})$ with $div~v=0$. If $\psi$ is a solution of
\begin{align}\label{eq4}
\left\{
\begin{array}{ll}
\psi_t+(v\cdot\nabla)\psi=div_{R}[-\sigma(v)\cdot{R}\psi+\frac {1} {2}\psi_{\infty}\nabla_{R}\frac{\psi}{\psi_{\infty}}+g]+f,  ~~~~~~~\\[1ex]
\psi|_{t=0}=\psi_0, \\[1ex]
\psi_{\infty}\nabla_{R}\frac{\psi}{\psi_{\infty}}\cdot{n}=0 ~~~~ \text{on} ~~~~ \partial B(0,1) .\\[1ex]
\end{array}
\right.
\end{align}
Then we have
\begin{align*}
\|\psi\|_{\tilde{L}_{T}^{\infty}(B^{\sigma}_{p,r}(\mathcal{L}^p))}+\frac 1 2 \|\psi\|_{\tilde{E}^{\sigma}_{p,r}(T)}\leq Ce^{CV(t)}\big(\|\psi_0\|_{B^{\sigma}_{p,r}(\mathcal{L}^p)}+(\int_{0}^{T} \|f(t)\|^{2}_{B^{\sigma}_{p,r}(\mathcal{L}^p)}+\|g(t)\|^{2}_{B^{\sigma}_{p,r}(\mathcal{L}^p)}dt)^{\frac 1 2} \big),
\end{align*}
with $V(t)=\int_{0}^{T}(\|v\|^2_{B^{s}_{p,r}}+1)dt$.
\end{lemm}
\begin{proof}
Applying $\Delta_j$ to \eqref{eq4}, we get
\begin{align}\label{eq5}
\left\{
\begin{array}{ll}
\partial_t\Delta_j\psi+(v\cdot\nabla)\Delta_j\psi=div_{R}[-\Delta_j(\sigma(v)\cdot{R}\psi)+\frac {1} {2}\psi_{\infty}\nabla_{R}\Delta_j\frac{\psi}{\psi_{\infty}}+\Delta_jg]+\Delta_jf+R_j,  ~~~~~~~\\[1ex]
\Delta_j\psi|_{t=0}=\Delta_j\psi_0, \\[1ex]
\psi_{\infty}\nabla_{R}\Delta_j\frac{\psi}{\psi_{\infty}}\cdot{n}=0 ~~~~ \text{on} ~~~~ \partial B(0,1) ,\\[1ex]
\end{array}
\right.
\end{align}
where $R_j=[v\cdot\nabla, \Delta_j]\psi$. We can define the flow $\Phi(t,x)$ of $v\in L^{\infty}_T(C^{0,1})$ such that
\begin{align}
\left\{
\begin{array}{ll}
\partial_t\Phi(t,x)=v(t,(\Phi(t,x))), \\[1ex]
\Phi(t,x)|_{t=0}=x.
\end{array}
\right.
\end{align}
Denote $\widetilde{A}(t,x,R)=A(t,\Phi(t,x),R)$ for any function A(t,x,R). Then we get the following equivalent equation
\begin{align}\label{eq6}
\left\{
\begin{array}{ll}
\partial_t\widetilde{\Delta_j\psi}=div_{R}[-\widetilde{\Delta_j(\sigma(v)\cdot{R}\psi)}+\frac {1} {2}\psi_{\infty}\nabla_{R}\widetilde{\Delta_j\frac{\psi}{\psi_{\infty}}}+\widetilde{\Delta_jg}]+\widetilde{\Delta_jf}+\widetilde{R_j},  ~~~~~~~\\[1ex]
\widetilde{\Delta_j\psi}|_{t=0}=\Delta_j\psi_0, \\[1ex]
\psi_{\infty}\nabla_{R}\widetilde{\Delta_j\frac{\psi}{\psi_{\infty}}}\cdot{n}=0 ~~~~ \text{on} ~~~~ \partial B(0,1) .\\[1ex]
\end{array}
\right.
\end{align}
Multiplying $sgn(\widetilde{\Delta_j\psi})|\frac {\widetilde{\Delta_j\psi}}{\psi_{\infty}}|^{p-1}$ to both sides of \eqref{eq6}. Since $div~v=0$, after some simple calculations, we obtain
\begin{align*}
     &\frac 1 p \partial_t\int_{R^d\times B}|\frac {\Delta_j\psi} {\psi_\infty}|^{p}\psi_{\infty}dxdR+\frac {p-1} {p^2} \int_{R^d\times B}|\nabla_R(\frac {\Delta_j\psi} {\psi_\infty})^{\frac p 2}|^{2}\psi_{\infty}dxdR\leq      \\
     & C\int_{R^d\times B}\big[\big(\Delta_j(\sigma(v)\cdot{R}\psi)\big)^2+|\Delta_jg|^2\big]|\frac {\Delta_j\psi} {\psi_\infty}|^{p-2}\psi_{\infty}dxdR+\int_{R^d\times B}|R_j+\Delta_jf||\frac {\Delta_j\psi} {\psi_\infty}|^{p-1}\psi_{\infty}dxdR.  \notag
\end{align*}
Using H\"{o}lder's inequality, we deduce that
\begin{align}\label{ineq1}
\partial_t\|\Delta_j\psi\|^{p}_{L^p_x(\mathcal{L}^p)}+\frac {1} {2} \int_{R^d\times B}|\nabla_R(\frac {\Delta_j\psi} {\psi_\infty})^{\frac p 2}|^{2}\psi_{\infty}dxdR\leq C[(\|\Delta_jf\|_{L^p_x(\mathcal{L}^p)}+\|R_j\|_{L^p_x(\mathcal{L}^p)})\|\Delta_j\psi\|^{p-1}_{L^p_x(\mathcal{L}^p)}   \notag
\\
+(\|\Delta_j(\sigma(v)\cdot{R}\psi)\|^2_{L^p_x(\mathcal{L}^p)}+\|\Delta_jg\|^2_{L^p_x(\mathcal{L}^p)})\|\Delta_j\psi\|^{p-2}_{L^p_x(\mathcal{L}^p)}].
\end{align}
Multiplying $\|\Delta_j\psi\|^{2-p}_{L^p_x(\mathcal{L}^p)}$ to \eqref{ineq1} and integrating over $[0,t]$, we obtain
\begin{align}\label{ineq2}
\|\Delta_j\psi\|^{2}_{L^p_x(\mathcal{L}^p)}\leq \|\Delta_j\psi_0\|^{2}_{L^p_x(\mathcal{L}^p)}+ C\int_{0}^{t}(\|\Delta_jf\|_{L^p_x(\mathcal{L}^p)}+\|R_j\|_{L^p_x(\mathcal{L}^p)})\|\Delta_j\psi\|_{L^p_x(\mathcal{L}^p)}   \notag
\\
+\|\Delta_j(\sigma(v)\cdot{R}\psi)\|^2_{L^p_x(\mathcal{L}^p)}+\|\Delta_jg\|^2_{L^p_x(\mathcal{L}^p)}dt'.
\end{align}
Applying Lemma \ref{new estimate}, we have
\begin{align}\label{ineq3}
\Big\|(2^{j\sigma}\|R_j\|_{L^p_{x}(\mathcal{L}^p)})_j\Big\|_{l^r(\mathbb{Z})}\leq C\|\nabla v\|_{B^{s-1}_{p,r}}\|\psi\|_{B^\sigma_{p,r}(\mathcal{L}^p)}.
\end{align}
Multiplying both sides of \eqref{ineq2} by $2^{2j\sigma}$, taking $L^{\infty}_T$ and $l^{\frac r 2}$ norm (for $r\geq 2$) and using \eqref{ineq3} and (2) of Lemma \ref{product} with $s_1=\sigma,\ s_2=s-1$, we get
\begin{align*}
\|\psi\|^{2}_{\tilde{L}_T^{\infty}({B^{\sigma}_{p,r}(\mathcal{L}^p)})}&\leq \|\psi_0\|^{2}_{L^p_x(\mathcal{L}^p)}+ C\int_{0}^{T}(\|f\|_{{B^{\sigma}_{p,r}(\mathcal{L}^p)}}+\|R_j\|_{{B^{\sigma}_{p,r}(\mathcal{L}^p)}})\|\psi\|_{{B^{\sigma}_{p,r}(\mathcal{L}^p)}}  \\
&+\|(\sigma(v)\cdot{R}\psi)\|^2_{{B^{\sigma}_{p,r}(\mathcal{L}^p)}}+\|g\|^2_{{B^{\sigma}_{p,r}(\mathcal{L}^p)}}dt  \\
&\leq  \|\psi_0\|^{2}_{L^p_x(\mathcal{L}^p)}+ C\int_{0}^{T}\|f\|^2_{{B^{\sigma}_{p,r}(\mathcal{L}^p)}}+\|g\|^2_{{B^{\sigma}_{p,r}(\mathcal{L}^p)}}
+(1+\|\nabla v\|^2_{{B^{s-1}_{p,r}(\mathcal{L}^p)}})\|\psi\|^{2}_{\tilde{L}_t^{\infty}({B^{\sigma}_{p,r}(\mathcal{L}^p)})}dt.
\end{align*}
Applying Gronwall's inequality, we deduce that
\begin{align}\label{ineq4}
\|\psi\|_{\tilde{L}_{T}^{\infty}(B^{\sigma}_{p,r}(\mathcal{L}^p))}\leq Ce^{CV(t)}\big(\|\psi_0\|_{B^{\sigma}_{p,r}(\mathcal{L}^p)}+(\int_{0}^{T} \|f(t)\|^{2}_{B^{\sigma}_{p,r}(\mathcal{L}^p)}+\|g(t)\|^{2}_{B^{\sigma}_{p,r}(\mathcal{L}^p)}dt)^{\frac 1 2} \big),
\end{align}
with $V(t)=\int_{0}^{T}(\|v\|^2_{B^{s}_{p,r}}+1)dt$. Integrating \eqref{ineq1} over $[0,T]$, we have
\begin{align}\label{ineq5}
\frac {1} {2} \int_{0}^{T}\int_{R^d\times B}|\nabla_R(\frac {\Delta_j\psi} {\psi_\infty})^{\frac p 2}|^{2}\psi_{\infty}dxdRdt&\leq \|\Delta_j\psi_0\|_{L^p_x(\mathcal{L}^p)}+C\int_{0}^{T}(\|\Delta_jf\|_{L^p_x(\mathcal{L}^p)}+\|R_j\|_{L^p_x(\mathcal{L}^p)})\|\Delta_j\psi\|^{p-1}_{L^p_x(\mathcal{L}^p)}   \notag \\
&+(\|\Delta_j(\sigma(v)\cdot{R}\psi)\|^2_{L^p_x(\mathcal{L}^p)}+\|\Delta_jg\|^2_{L^p_x(\mathcal{L}^p)})\|\Delta_j\psi\|^{p-2}_{L^p_x(\mathcal{L}^p)}dt.
\end{align}
Multiplying both sides of \eqref{ineq5} by $2^{2pj\sigma}$, taking $l^{\frac r p}$ norm (for $r\geq p$) and using \eqref{ineq4}, we obtain
\begin{align*}
\frac 1 2 \|\psi\|_{\tilde{E}^{\sigma}_{p,r}(T)}\leq Ce^{CV(t)}\big(\|\psi_0\|^p_{B^{\sigma}_{p,r}(\mathcal{L}^p)}+(\int_{0}^{T} \|f(t)\|^{2}_{B^{\sigma}_{p,r}(\mathcal{L}^p)}+\|g(t)\|^{2}_{B^{\sigma}_{p,r}(\mathcal{L}^p)}dt)^{\frac 1 2} \big).
\end{align*}
We thus complete the proof.
\end{proof}

\begin{prop}\label{existence1}
Let $s>1+max(\frac d p,\frac 1 2)$, $2\leq p\leq r<\infty$. Assume that $\psi_0\in B^{s-1}_{p,r}(\mathcal{L}^p)$, and $v\in C_T(B^{s}_{p,r})$ with $div~v=0$. Then
\begin{align}\label{eq7}
\left\{
\begin{array}{ll}
\psi_t+(v\cdot\nabla)\psi=div_{R}[-\sigma(v)\cdot{R}\psi+\frac {1} {2}\psi_{\infty}\nabla_{R}\frac{\psi}{\psi_{\infty}}],  ~~~~~~~\\[1ex]
\psi|_{t=0}=\psi_0, \\[1ex]
\psi_{\infty}\nabla_{R}\frac{\psi}{\psi_{\infty}}\cdot{n}=0 ~~~~ \text{on} ~~~~ \partial B(0,1) ,\\[1ex]
\end{array}
\right.
\end{align}
has a unique solution $\psi\in C_T(B^{s-1}_{p,r}(\mathcal{L}^p))$. Moreover, we have the following estimate
\begin{align*}
\|\psi\|_{\tilde{L}_{T}^{\infty}(B^{s-1}_{p,r}(\mathcal{L}^p))}+\frac 1 2 \|\psi\|_{\tilde{E}^{s-1}_{p,r}(T)}\leq Ce^{CV(t)}\|\psi_0\|_{B^{s-1}_{p,r}(\mathcal{L}^p)},
\end{align*}
with $V(t)=\int_{0}^{T}(\|v\|^2_{B^{s}_{p,r}}+1)dt$.
\end{prop}
\begin{proof}
Similar as in Lemma \ref{estimate1}, we denote $\widetilde{A}(t,x,R)=A(t,\Phi(t,x),R)$ for any function $A(t,x,R)$. Then we get the following equivalent equation
\begin{align}
\left\{
\begin{array}{ll}
\partial_t\widetilde{\psi}=div_{R}[-\widetilde{(\sigma(v)\cdot{R}\psi)}+\frac {1} {2}\psi_{\infty}\nabla_{R}\widetilde{\frac{\psi}{\psi_{\infty}}}],  ~~~~~~~\\[1ex]
\widetilde{\psi}|_{t=0}=\psi_0, \\[1ex]
\psi_{\infty}\nabla_{R}\widetilde{\frac{\psi}{\psi_{\infty}}}\cdot{n}=0 ~~~~ \text{on} ~~~~ \partial B(0,1) .\\[1ex]
\end{array}
\right.
\end{align}
Applying Lemma \ref{weak solution} for each fixed x, we deduce that the existence and uniqueness of $\psi(t,x,R)\in C([0,T];\mathcal{L}^p)$. By Lemma \ref{estimate1}, we get the following estimate
\begin{align*}
\|\psi\|_{\tilde{L}_{T}^{\infty}(B^{s-1}_{p,r}(\mathcal{L}^p))}+\frac 1 2 \|\psi\|_{\tilde{E}^{s-1}_{p,r}(T)}\leq Ce^{CV(t)}\|\psi_0\|_{B^{s-1}_{p,r}(\mathcal{L}^p)}.
\end{align*}
Since the equation \eqref{eq7} is a linear equation, we deduce that solution $\psi(t,x,R)\in L^{\infty}_T(B^{s-1}_{p,r}(\mathcal{L}^p))$ is unique.
We only need to check $\psi\in C_T(B^{s-1}_{p,r}(\mathcal{L}^p)).$ By a similar calculation as in Lemma \ref{estimate1} with $f=g=0$, we deduce that
\begin{align}\label{C}
\partial_t\|\Delta_j\psi\|^{2}_{L^p_x(\mathcal{L}^p)}&\leq C(\|R_j\|_{L^p_x(\mathcal{L}^p)}\|\Delta_j\psi\|_{L^p_x(\mathcal{L}^p)}
+\|\Delta_j(\sigma(v)\cdot{R}\psi)\|^2_{L^p_x(\mathcal{L}^p)})  \notag\\
&\leq Cc^2_j 2^{-2j(s-1)}(\|v\|^2_{B^{s}_{p,r}}+1)\|\psi\|^2_{B^{s-1}_{p,r}(\mathcal{L}^p)},
\end{align}
where $(c_j)_{j\geq-1}\in l^r$.
Since $r<\infty$ and $\psi(t,x,R)\in L^{\infty}_T(B^{s-1}_{p,r}(\mathcal{L}^p))$, for any $t_1< t_2\in [0,T]$, for any $\epsilon_1>0$, there exists N such that
$$\sup_{t\in [0,T]} (\sum_{j\geq N}2^{j(s-1)r}\|\psi\|^r_{L_x^p(\mathcal{L}^p)})^{\frac 1 r }\leq \frac {\epsilon_1} {4} .$$
Hence
\begin{align*}
\|\psi(t_1)-\psi(t_2)\|_{B^{s-1}_{p,r}(\mathcal{L}^p)}&\leq (\sum_{-1\leq j<N}2^{j(s-1)r}\|\Delta_j\psi(t_1)-\Delta_j\psi(t_2)\|^r_{L_x^p(\mathcal{L}^p)})^{\frac 1 r }+2\sup_{t\in [0,T]} (\sum_{j\geq N}2^{jsr}\|\Delta_j\psi\|^r_{L_x^p(\mathcal{L}^p)})^{\frac 1 r }   \\
&\leq (\sum_{-1\leq j<N}2^{j(s-1)r}\int_{t_1}^{t_2}\partial_t\|\Delta_j\psi\|^{r}_{L^p_x(\mathcal{L}^p)}dt')^{\frac 1 r }+\frac {\epsilon_1} {2}  \\
&\leq (\sum_{-1\leq j<N}\frac r 2 2^{j(s-1)r}\int_{t_1}^{t_2}\partial_t\|\Delta_j\psi\|^{2}_{L^p_x(\mathcal{L}^p)}
\|\Delta_j\psi\|^{r-2}_{L^p_x(\mathcal{L}^p)}dt')^{\frac 1 r }+\frac {\epsilon_1} {2}  \\
&\leq (\sum_{-1\leq j<N}\frac r 2 2^{j(s-2)r}\int_{t_1}^{t_2}c^2_j(\|v\|^2_{B^{s}_{p,r}}+1)\|\psi\|^2_{B^{s-1}_{p,r}(\mathcal{L}^p)}
\|\Delta_j\psi\|^{r-2}_{L^p_x(\mathcal{L}^p)}dt')^{\frac 1 r }+\frac {\epsilon_1} {2}  \\
&\leq C(N+1)(\int_{t_1}^{t_2}\|v\|^2_{B^{s}_{p,r}}+1dt')^{\frac 1 r }+\frac {\epsilon_1} {2}.
\end{align*}
Then we have
$$\|\psi(t_1)-\psi(t_2)\|_{B^{s-1}_{p,r}(\mathcal{L}^p)}\rightarrow 0~~ as~~t_1\rightarrow t_2,$$
from which we know that $\psi\in C_T(B^{s-1}_{p,r}(\mathcal{L}^p)).$

\end{proof}

\begin{lemm}\label{estimate2}
Let $0<\nu=\frac {\epsilon} {Re}\leq 1$, $s>1+\frac d p$, $(p-1)k>1$, $s-1\leq\sigma\leq s$, $1\leq r<\infty$, $1<p<\infty$ and $1\leq \rho\leq\infty$. Assume that $u_0\in B^{\sigma}_{p,r}$, $\psi\in \tilde{L}^{\infty}_{T}(B^{\sigma-1}_{p,r})(\mathcal{L}^p)$, $f\in \tilde{L}^{\rho}_{T}(B^{\sigma-2+\frac {2} {\rho}}_{p,r})$ and $v\in L^{\infty}_T(B^{s}_{p,r})$ with $div~v=0$. If $u$ is a solution of
\begin{align}\label{eq8}
\left\{
\begin{array}{ll}
u_t+(v\cdot\nabla)u-\frac {\epsilon} {Re}\Delta u+\nabla{P}=\frac {1-\epsilon} {Re}div~\tau+f,  ~~~~~~~div~u=0,\\[1ex]
\tau_{ij}=\epsilon\int_{B}(R_{i}\nabla_{R_j}\mathcal{U})\psi dR, \\[1ex]
u|_{t=0}=u_0, \\[1ex]
\end{array}
\right.
\end{align}
then we have
\begin{align*}
\|u\|_{\tilde{L}_{T}^{\infty}(B^{\sigma}_{p,r})}\leq Ce^{CV(t)}\big(\|u_0\|_{B^{\sigma}_{p,r}}+\nu^{\frac 1 \rho -1}
\|f\|_{\tilde{L}^{\rho}_{T}(B^{\sigma-2+\frac 2 \rho}_{p,r})}+\|\psi\|_{\tilde{L}^{\infty}_{T}(B^{\sigma-1}_{p,r})(\mathcal{L}^p)}\big),
\end{align*}
with $V(t)=\int_{0}^{T}(\|v\|^2_{B^{s}_{p,r}}+1)dt$.
Moreover, if $f=0$ and $\sigma=s$, then we have $u\in C_T(B^{s}_{p,r})$.
\end{lemm}
\begin{proof}
Taking $div$ for \eqref{eq8}, we have
$$\Delta P=div(\frac {1-\epsilon} {Re}div~\tau+f-(v\cdot\nabla)u).$$
Thus $P$ is a solution of an elliptic equation and we can obtain
$$\nabla P=\nabla(\Delta)^{-1}div(\frac {1-\epsilon} {Re}div~\tau+f-(v\cdot\nabla)u).$$
We rewrite \eqref{eq8} as
\begin{align*}
u_t+(v\cdot\nabla)u-\frac {\epsilon} {Re}\Delta u=-\nabla(\Delta)^{-1}div(\frac {1-\epsilon} {Re}div~\tau+f-(v\cdot\nabla)u)+\frac {1-\epsilon} {Re}div~\tau+f.
\end{align*}
Since $\nabla(\Delta)^{-1}div$ is a Calderon-Zygmund operator and $1<p<\infty$, applying Lemma \ref{u} and Lemma \ref{P}, we deduce that
\begin{align*}
\|u\|_{\tilde{L}_{T}^{\infty}(B^{\sigma}_{p,r})}&\leq Ce^{CV_p(v,t)}\big(\|u_0\|_{B^{\sigma}_{p,r}}+\|\nabla(\Delta)^{-1}div(v\cdot\nabla u)\|_{\tilde{L}^{1}_{T}(B^{\sigma}_{p,r})}+\nu^{\frac 1 \rho -1}\|f\|_{\tilde{L}^{\rho}_{T}(B^{\sigma-2+\frac 2 \rho}_{p,r})}+\frac 1 {\epsilon} \|div~\tau\|_{\tilde{L}^{\infty}_{T}(B^{\sigma-2}_{p,r})}\big)  \\
&\leq Ce^{CV_p(v,t)}\big(\|u_0\|_{B^{\sigma}_{p,r}}+\int_{0}^{T}\|u\|_{\tilde{L}_{t}^{\infty}(B^{\sigma}_{p,r})}\|v\|_{B^{s}_{p,r}}dt+\nu^{\frac 1 \rho -1}\|f\|_{\tilde{L}^{\rho}_{T}(B^{\sigma-2+\frac 2 \rho}_{p,r})}+\frac 1 {\epsilon} \|div~\tau\|_{\tilde{L}^{\infty}_{T}(B^{\sigma-2}_{p,r})}\big).
\end{align*}
Since $V_p(v,t)\leq \int_{0}^{T}\|v\|_{B^{s}_{p,r}}dt\leq V(t)=\int_{0}^{T}(\|v\|^2_{B^{s}_{p,r}}+1)dt$, applying Gronwall's inequality, we obtain
\begin{align*}
\|u\|_{\tilde{L}_{T}^{\infty}(B^{\sigma}_{p,r})}&\leq Ce^{CV(t)}\big(\|u_0\|_{B^{\sigma}_{p,r}}+\nu^{\frac 1 \rho -1}\|f\|_{\tilde{L}^{\rho}_{T}(B^{\sigma-2+\frac 2 \rho}_{p,r})}+\frac 1 {\epsilon} \|div~\tau\|_{\tilde{L}^{\infty}_{T}(B^{\sigma-2}_{p,r})}\big).
\end{align*}
Applying Lemma \ref{Lemma3}, we have $\|\Delta_j \tau\|_{L^p}\leq\|\Delta_j \psi\|_{L^p(\mathcal{L}^p)}$. Then we can easily deduce that
$$\frac 1 {\epsilon} \|div~\tau\|_{\tilde{L}^{\infty}_{T}(B^{\sigma-2}_{p,r})}\leq \|\psi\|_{\tilde{L}^{\infty}_{T}(B^{\sigma-1}_{p,r})(\mathcal{L}^p)},$$
which implies that
\begin{align*}
\|u\|_{\tilde{L}_{T}^{\infty}(B^{\sigma}_{p,r})}\leq Ce^{CV(t)}\big(\|u_0\|_{B^{\sigma}_{p,r}}+\nu^{\frac 1 \rho -1}
\|f\|_{\tilde{L}^{\rho}_{T}(B^{\sigma-2+\frac 2 \rho}_{p,r})}+\|\psi\|_{\tilde{L}^{\infty}_{T}(B^{\sigma-1}_{p,r})(\mathcal{L}^p)}\big).
\end{align*}
Let $\sigma=s$ and $f=0$. We proved that $u\in L_{T}^{\infty}(B^{s}_{p,r})$, $\tau\in L_{T}^{\infty}(B^{s-2}_{p,r})$. Applying Lemma \ref{product}, we have $v\cdot\nabla u\in L_{T}^{\infty}(B^{s-1}_{p,r})$, $\nabla P\in L_{T}^{\infty}(B^{s-2}_{p,r})$ and $\Delta u\in L_{T}^{\infty}(B^{s-2}_{p,r})$. So from the equation \eqref{eq8}, we deduce that $\partial_t u\in L_{T}^{\infty}(B^{s-2}_{p,r})$, then $u\in C_{T}(B^{s-2}_{p,r})$. An interpolation argument ensures that $u\in C_{T}(B^{s'}_{p,r})$, for any $s'<s$. Since $r<\infty$ and $u\in L_{T}^{\infty}(B^{s}_{p,r})$, for any $t_1, t_2\in [0,T]$, for any $\epsilon_1>0$, there exists N such that
$$\sup_{t\in [0,T]} (\sum_{j\geq N}2^{jsr}\|u\|^r_{L^p})^{\frac 1 r }\leq \frac {\epsilon_1} {4} .$$
Hence
\begin{align*}
\|u(t_1)-u(t_2)\|_{B^{s}_{p,r}}&\leq (\sum_{-1\leq j<N}2^{jsr}\|u(t_1)-u(t_2)\|^r_{L^p})^{\frac 1 r }+2\sup_{t\in [0,T]} (\sum_{j\geq N}2^{jsr}\|u\|^r_{L^p})^{\frac 1 r }   \\
&\leq 2^N(\sum_{-1\leq j<N}2^{j(s-1)r}\|u(t_1)-u(t_2)\|^r_{L^p})^{\frac 1 r }+\frac {\epsilon_1} {2}  \\
&\leq 2^N\|u(t_1)-u(t_2)\|_{B^{s-1}_{p,r}}+\frac {\epsilon_1} {2} ,
\end{align*}
from which we know that $u\in C_T(B^{s}_{p,r})$.
We thus complete the proof.
\end{proof}

\section{Local well-posedness}
\par
In this section, we will investigate the local well-posedness for \eqref{eq1} with $0\leq\nu\leq 1$ in Besov spaces.
\begin{prop}\label{Local}
Let $d\geq 2$. Assume that $\nu \in [0,1]$, $s>1+max\{\frac 1 2, \frac d p\}$, $2\leq p\leq r<\infty$ and $k(p-1)>1$. Then for any $K>0$, $(u_0,\psi_0)\in B_K=\{(v,\phi)\in B^s_{p,r}\times B^{s-1}_{p,r}(\mathcal{L}^p):\|v\|_{B^s_{p,r}}+\|\phi\|_{B^{s-1}_{p,r}(\mathcal{L}^p)}\leq K, div~v=0\}$, there exists $T=T(K,s,p,r,d)>0$ such that the equation \eqref{eq1} has a unique solution $H^{\nu}(u_0,\psi_0)=(u^{\nu},\psi^{\nu})\in C_T(B^s_{p,r})\times C_T( B^{s-1}_{p,r}(\mathcal{L}^p))$. Moreover, $\psi^{\nu}\in \tilde{E}^{s-1}_{p,r}(T)$ and there exists $C=C(K,s,p,r,d)>0$ such that
\begin{align}
\|H^{\nu}(u_0,\psi_0)\|_{\tilde{L}_T^{\infty}(B^s_{p,r})\times\tilde{L}_T^{\infty}(B^{s-1}_{p,r}(\mathcal{L}^p))}\leq C, \quad \forall\nu\in[0,1].
\end{align}
Moreover, if $(u_0,\psi_0)\in B^{\gamma}_{p,r}\times B^{\gamma-1}_{p,r}(\mathcal{L}^p)$ for any $\gamma>s$, then there exists $C_1=C_1(K,\gamma,s,p,r,d)>0$ such that
\begin{align}
\|H^{\nu}(u_0,\psi_0)\|_{\tilde{L}_T^{\infty}(B^{\gamma}_{p,r})\times\tilde{L}_T^{\infty}(B^{\gamma-1}_{p,r}(\mathcal{L}^p))}\leq C_1(\|u_{0}\|_{B^{\gamma}_{p,r}}+\|\psi_{0}\|_{B^{\gamma-1}_{p,r}(\mathcal{L}^p)}), \quad \forall\nu\in[0,1].
\end{align}
\end{prop}
\begin{proof}
If $\nu=0$, since the equation \eqref{eq1} is linear and can be handled more simply, we omit the details. Let $\nu>0$. We divide it into four steps to prove Proposition \ref{Local}.

\textbf{Step one. Approximate solutions.} \\
First, we construct approximate solutions of some linear equations.   \\
Let $(u^0,\psi^0)\triangleq(S_{0}u_0,S_{0}\psi_0)$. We define a sequence $(u^n,\psi^n)_{n\in\mathbb{N}}$ by solving the following linear equations:
\begin{align}\label{eq9}
\left\{
\begin{array}{ll}
\partial_t u^{n+1}+(u^{n}\cdot\nabla)u^{n+1}-\frac {\epsilon} {Re}\Delta u^{n+1}+\nabla{P^{n+1}}=\frac {1-\epsilon} {Re}div~\tau^{n},  ~~~~~~~div~u^{n+1}=0,\\[1ex]
\partial_t\psi^{n+1}+(u^{n}\cdot\nabla)\psi^{n+1}=div_{R}[-\sigma(u^{n})\cdot{R}\psi^{n+1}+\frac {1} {2}\psi_{\infty}\nabla_{R}\frac{\psi^{n+1}}{\psi_{\infty}}],  ~~~~~~~\\[1ex]
\tau^{n}_{ij}=\epsilon\int_{B}(R_{i}\nabla_{R_j}\mathcal{U})\psi^{n} dR, \\[1ex]
u^{n+1}|_{t=0}=S_{n+1}u_0,~ \psi^{n+1}|_{t=0}=S_{n+1}\psi_0, \\[1ex]
\psi_{\infty}\nabla_{R}\frac{\psi^{n+1}}{\psi_{\infty}}\cdot{n}=0 ~~~~ \text{on} ~~~~ \partial B(0,1) .\\[1ex]
\end{array}
\right.
\end{align}
Assume that $u^n\in C_T(B^{s}_{p,r})$ and $\psi^n\in C_T(B^{s-1}_{p,r}(\mathcal{L}^p))$ for any positive T. By Lemma \ref{estimate2}, we obtain $u^{n+1}\in C_T(B^{s}_{p,r})$. According to Proposition \ref{existence1}, we have  $\psi^{n+1}\in C_T(B^{s-1}_{p,r}(\mathcal{L}^p))$.

\textbf{Step two. Uniform bounds.}  \\
Next, we want to find some positive $T$ such that for which the approximate solutions are uniformly bounded. \\
By Lemma \ref{estimate2}, we obtain
\begin{align}\label{Bound1}
\|u^{n+1}\|_{\tilde{L}_{T}^{\infty}(B^{s}_{p,r})}\leq Ce^{CU^n(t)}\big(\|u_0\|_{B^{s}_{p,r}}+\|\psi^n\|_{\tilde{L}^{\infty}_{T}(B^{s-1}_{p,r})(\mathcal{L}^p)}\big),
\end{align}
with $U^n(t)=\int_{0}^{T}(\|u^n\|^2_{B^{s}_{p,r}}+1)dt$.
According to Proposition \ref{existence1}, we have
\begin{align}\label{Bound2}
\|\psi^{n+1}\|_{\tilde{L}_{T}^{\infty}(B^{s-1}_{p,r}(\mathcal{L}^p))}\leq Ce^{CU^n(t)}\|\psi_0\|_{B^{s-1}_{p,r}(\mathcal{L}^p)}.
\end{align}
Fix a $T>0$ such that $T\leq\frac {1} {C^2+81C^5(\|u_0\|_{B^{s}_{p,r}}+\|\psi_0\|_{B^{s-1}_{p,r}(\mathcal{L}^p)})^2}$. We claim that for any $n$ and $t\in[0,T]:$
\begin{align}\label{Bound3}
\|u^{n}\|_{\tilde{L}_{T}^{\infty}(B^{s}_{p,r})}\leq 9C^2(\|u_0\|_{B^{s}_{p,r}}+\|\psi_0\|_{B^{s-1}_{p,r}(\mathcal{L}^p)}), ~\|\psi^{n}\|_{\tilde{L}_{T}^{\infty}(B^{s-1}_{p,r}(\mathcal{L}^p))}\leq 3C\|\psi_0\|_{B^{s-1}_{p,r}(\mathcal{L}^p)}.
\end{align}

By induction, when $n=0$, \eqref{Bound3} holds true. Now suppose that we can prove \eqref{Bound3} holds true for $n$. Plugging \eqref{Bound3} into \eqref{Bound2}, then we have
\begin{align}\label{Bound4}
\|\psi^{n+1}\|_{\tilde{L}_{T}^{\infty}(B^{s-1}_{p,r}(\mathcal{L}^p))}\leq 3C\|\psi_0\|_{B^{s-1}_{p,r}(\mathcal{L}^p)}
\end{align}
Plugging \eqref{Bound3} into \eqref{Bound1}, then we get
$\|u^{n+1}\|_{\tilde{L}_{T}^{\infty}(B^{s}_{p,r})}\leq 9C^2(\|u_0\|_{B^{s}_{p,r}}+\|\psi_0\|_{B^{s-1}_{p,r}(\mathcal{L}^p)}).$

Therefore, there exists $C=C(K,s,p,r,d)>0$ such that
\begin{align}\label{Bound5}
\|(u^n,\psi^n)\|_{\tilde{L}_T^{\infty}(B^s_{p,r})\times\tilde{L}_T^{\infty}(B^{s-1}_{p,r}(\mathcal{L}^p))}\leq C,   ~~\forall~n.
\end{align}
\textbf{Step three. Convergence.}  \\
We are going to prove that $(u^n,\psi^n)$ is a Cauchy sequence in $\tilde{L}_T^{\infty}(B^{s-1}_{p,r})\times\tilde{L}_T^{\infty}(B^{s-2}_{p,r}(\mathcal{L}^p))$. According to \eqref{eq9}, we have
\begin{align}\label{eq10}
\left\{
\begin{array}{ll}
\partial_t (u^{n+1}-u^n)+(u^{n}\cdot\nabla)(u^{n+1}-u^n)-\frac {\epsilon} {Re}\Delta (u^{n+1}-u^n)+\nabla(P^{n+1}-P^n)=\frac {1-\epsilon} {Re}div~(\tau^{n}-\tau^{n-1})+F^n,  \\[1ex]
\partial_t(\psi^{n+1}-\psi^{n})+(u^{n}\cdot\nabla)(\psi^{n+1}-\psi^{n})=div_{R}[-\sigma(u^{n})\cdot{R}(\psi^{n+1}-\psi^{n})+\frac {1} {2}\psi_{\infty}\nabla_{R}\frac{(\psi^{n+1}-\psi^{n})}{\psi_{\infty}}+g^n]+f^n,  ~~~~~~~\\[1ex]
u^{n+1}-u^n|_{t=0}=\Delta_{n}u_0,~ \psi^{n+1}-\psi^n|_{t=0}=\Delta_{n}\psi_0, \\[1ex]
\end{array}
\right.
\end{align}
where $F^n=-(u^{n}-u^{n-1})\nabla u^n$, $g^n=-\sigma(u^{n}-u^{n-1})R \psi^n$, $f^n=-(u^{n}-u^{n-1})\nabla \psi^n$
and $\nabla (P^{n+1}-P^n)=\nabla(\Delta)^{-1}div(\frac {1-\epsilon} {Re}div~(\tau^{n}-\tau^{n-1})+F^n-(u^{n}\cdot\nabla)(u^{n+1}-u^n)).$ \\

By a similar calculation as in Lemma \ref{estimate2} and uniform bounds for $u^n$, we obtain
\begin{align*}
\|u^{n+1}-u^n\|_{\tilde{L}_{T}^{\infty}(B^{s-1}_{p,r})}&\leq C\big(\|\Delta_{n}u_0\|_{B^{s-1}_{p,r}}+
\|(u^{n}-u^{n-1})\nabla u^n\|_{L^{1}_{T}(B^{s-1}_{p,r})}+\|\psi^{n}-\psi^{n-1}\|_{\tilde{L}^{\infty}_{T}(B^{s-2}_{p,r})(\mathcal{L}^p)}\big)  \\
&\leq C\big(\|\Delta_{n}u_0\|_{B^{s-1}_{p,r}}+
T\|u^{n}-u^{n-1} \|_{\tilde{L}_{T}^{\infty}(B^{s-1}_{p,r})}+\|\psi^{n}-\psi^{n-1}\|_{\tilde{L}^{\infty}_{T}(B^{s-2}_{p,r})(\mathcal{L}^p)}\big).
\end{align*}
By Lemma \ref{estimate1}, Lemma \ref{product} with $s_1=s-2,\ s_2=s-1$ and uniform bounds for $u^n$ and $\psi^n$, we obtain
\begin{align*}
\|\psi^{n+1}-\psi^{n}\|_{\tilde{L}_{T}^{\infty}(B^{s-2}_{p,r}(\mathcal{L}^p))}&\leq C\big(\|\Delta_{n}\psi_0\|_{B^{s-2}_{p,r}(\mathcal{L}^p)}+(\int_{0}^{T} \|f^n\|^2_{\tilde{L}_{T}^{\infty}(B^{s-2}_{p,r})}+\|g^n \|^2_{\tilde{L}_{T}^{\infty}(B^{s-2}_{p,r})} dt)^{\frac 1 2 }\big)   \\
&\leq C\big(\|\Delta_{n}\psi_0\|_{B^{s-2}_{p,r}(\mathcal{L}^p)}+T^{\frac 1 2 } \|(u^{n}-u^{n-1})\|_{\tilde{L}_{T}^{\infty}(B^{s-1}_{p,r})} \big).
\end{align*}
By a direct calculation, we get
\begin{align*}
\|\Delta_{n}u_0\|_{B^{s-1}_{p,r}}\leq  (\sum_{|j-n|\leq 1}2^{jr(s-1)}\|\Delta_j\Delta_{n}u_0\|^r_{L^p})^{\frac 1 r} \leq C2^{-n}\|u_0\|_{B^{s}_{p,r}}.
\end{align*}
Similar, we have $\|\Delta_{n}\psi_0\|_{B^{s-2}_{p,r}(\mathcal{L}^p)}\leq C2^{-n}\|\psi_0\|_{B^{s-1}_{p,r}(\mathcal{L}^p)}.$
Define $A_n=\|u^{n+1}-u^n\|_{\tilde{L}_{T}^{\infty}(B^{s-1}_{p,r})}$ and $B_n=\|\psi^{n+1}-\psi^{n}\|_{\tilde{L}_{T}^{\infty}(B^{s-2}_{p,r}(\mathcal{L}^p))}$.
Then we deduce that
\begin{align}\label{conver1}
A_n\leq C\big(2^{-n}\|u_0\|_{B^{s}_{p,r}}+
TA_{n-1}+B_{n-1}\big)
\end{align}
and
\begin{align}\label{conver2}
B_n\leq C\big(2^{-n}\|\psi_0\|_{B^{s-1}_{p,r}(\mathcal{L}^p)}+T^{\frac 1 2 }A_{n-1}\big).
\end{align}
Plugging \eqref{conver2} into \eqref{conver1} and choosing a sufficiently small $T$, we obtain
\begin{align}\label{conver3}
A_n\leq C2^{-n}(\|u_0\|_{B^{s}_{p,r}}+\|\psi_0\|_{B^{s-1}_{p,r}(\mathcal{L}^p)})
+\frac 1 4 (A_{n-1}+A_{n-2}).
\end{align}
If $n<0$, we can set $A_n=0$. Then we deduce that
\begin{align}
\sum_{n=0}^{m}A_n\leq C\sum_{n=0}^{m}2^{-n}(\|u_0\|_{B^{s}_{p,r}}+\|\psi_0\|_{B^{s-1}_{p,r}(\mathcal{L}^p)})
+\frac 1 4 \sum_{n=0}^{m}(A_{n-1}+A_{n-2}),
\end{align}
and
\begin{align*}
\frac 1 4 \sum_{n=0}^{m}(A_{n-1}+A_{n-2})\leq\frac 1 2 \sum_{n=0}^{m}A_n.
\end{align*}
Then we have
\begin{align}
\sum_{n=0}^{m}A_n\leq C\sum_{n=0}^{m}2^{-n}(\|u_0\|_{B^{s}_{p,r}}+\|\psi_0\|_{B^{s-1}_{p,r}(\mathcal{L}^p)}),
\end{align}
which implies that $\sum_{n=0}^{\infty}A_n$ is convergent. We know $\sum_{n=0}^{m}B_n$ is also convergent by \eqref{conver2}. Hence, we deduce that $(u^n,\psi^n)$ is a Cauchy sequence in $\tilde{L}_T^{\infty}(B^{s-1}_{p,r})\times\tilde{L}_T^{\infty}(B^{s-2}_{p,r}(\mathcal{L}^p))$. Thus, there exists $(u,\psi)\in \tilde{L}_T^{\infty}(B^{s-1}_{p,r})\times\tilde{L}_T^{\infty}(B^{s-2}_{p,r}(\mathcal{L}^p))$ such that
$$u^n\rightarrow u~~in~~\tilde{L}_T^{\infty}(B^{s-1}_{p,r})~~and ~~\psi^n\rightarrow \psi~~in~~\tilde{L}_T^{\infty}(B^{s-2}_{p,r})(\mathcal{L}^p). $$

Uniform bounds for $u^n$ and $\psi^n$ and the Fatou property for Besov spaces ensure that $(u,\psi)\in \tilde{L}_T^{\infty}(B^{s}_{p,r})\times\tilde{L}_T^{\infty}(B^{s-1}_{p,r}(\mathcal{L}^p))$. Moreover, there exists $C=C(K,s,p,r,d)>0$ such that
\begin{align}\label{Bound6}
\|(u,\psi)\|_{\tilde{L}_T^{\infty}(B^s_{p,r})\times\tilde{L}_T^{\infty}(B^{s-1}_{p,r}(\mathcal{L}^p))}\leq C,   ~~\forall~\nu\in[0,1].
\end{align}
An interpolation argument ensures that
$$u^n\rightarrow u~~in~~\tilde{L}_T^{\infty}(B^{s'}_{p,r})~~and ~~\psi^n\rightarrow \psi~~in~~\tilde{L}_T^{\infty}(B^{s'-1}_{p,r})(\mathcal{L}^p). $$
Passing to the limit in \eqref{eq9} in the weak sense, we can prove that $(u,\psi)$ is a solution of \eqref{eq1}.
By a similar calculation as in Lemma\ref{estimate2} and Lemma\ref{estimate1}, we deduce that $(u,\psi)\in C_{T}(B^{s}_{p,r})\times C_{T}(B^{s-1}_{p,r}(\mathcal{L}^p))$.
Moreover, if $(u_0,\psi_0)\in B^{\gamma}_{p,r}\times B^{\gamma-1}_{p,r}(\mathcal{L}^p)$ for any $\gamma>s$, similar to step one, then there exists $C_1=C_1(K,\gamma,s,p,r,d)>0$ such that
\begin{align}\label{Bound7}
\|H^{\nu}(u_0,\psi_0)\|_{\tilde{L}_T^{\infty}(B^{\gamma}_{p,r})\times\tilde{L}_T^{\infty}(B^{\gamma-1}_{p,r}(\mathcal{L}^p))}\leq C_1(\|u_{0}\|_{B^{\gamma}_{p,r}}+\|\psi_{0}\|_{B^{\gamma-1}_{p,r}(\mathcal{L}^p)}), \quad \forall\nu\in[0,1].
\end{align}

\textbf{Step four. Uniqueness.}  \\

Let $(u,\psi)$ and $(v,\phi)$ be two solutions of \eqref{eq1} with the same initial data. Then we get
\begin{align}\label{eq11}
\left\{
\begin{array}{ll}
\partial_t (u-v)+(u\cdot\nabla)(u-v)-\frac {\epsilon} {Re}\Delta (u-v)+\nabla(P_1-P_2)=\frac {1-\epsilon} {Re}div~(\tau_1-\tau_2)+F,  \\[1ex]
\partial_t(\psi-\phi)+(u\cdot\nabla)(\psi-\phi)=div_{R}[-\sigma(u)\cdot{R}(\psi-\phi)+\frac {1} {2}\psi_{\infty}\nabla_{R}\frac{(\psi-\phi)}{\psi_{\infty}}+g]+f,  ~~~~~~~\\[1ex]
u-v|_{t=0}=0,~ \psi-\phi|_{t=0}=0, \\[1ex]
\end{array}
\right.
\end{align}
where $F=-(u-v)\nabla v$, $g=-\sigma(u-v)R \phi$, $f=-(u-v)\nabla \phi$ and $\nabla (P_1-P_2)=\nabla(\Delta)^{-1}div(\frac {1-\epsilon} {Re}div~(\tau_1-\tau_2)+F-(u\cdot\nabla)(u-v)).$ \\
By a similar calculation as in Step three and uniform bounds for $u$ and $v$, we obtain
\begin{align}\label{Unique1}
\|u-v\|_{\tilde{L}_{T}^{\infty}(B^{s-1}_{p,r})}&\leq C\big(
\|(u-v)\nabla v\|_{L^{1}_{T}(B^{s-1}_{p,r})}+\|\psi-\phi\|_{\tilde{L}^{\infty}_{T}(B^{s-2}_{p,r})(\mathcal{L}^p)}\big)  \notag \\
&\leq C\big(
T\|u-v \|_{\tilde{L}_{T}^{\infty}(B^{s-1}_{p,r})}+\|\psi-\phi\|_{\tilde{L}^{\infty}_{T}(B^{s-2}_{p,r})(\mathcal{L}^p)}\big).
\end{align}
By Lemma \ref{estimate1}, Lemma \ref{product} with $s_1=s-2,\ s_2=s-1$ and uniform bounds for $u$, $v$, $\psi$ and $\phi$, we have
\begin{align}\label{Unique2}
\|\psi-\phi\|_{\tilde{L}_{T}^{\infty}(B^{s-2}_{p,r}(\mathcal{L}^p))}\leq C\big(T^{\frac 1 2 } \|u-v \|_{\tilde{L}_{T}^{\infty}(B^{s-1}_{p,r})} \big).
\end{align}
Plugging \eqref{Unique2} into \eqref{Unique1} and choosing $T_1$ is small enough, $\forall t\in[0,T_1]$, we can easily deduct that $u=v$ for $a.e.(t,x)$.
From \eqref{Unique2}, we get $\psi=\phi$ for $a.e.(t,x,R)$. By a standard connectivity argument, we can prove the uniqueness on $[0,T]$.
\end{proof}

\section{Inviscid limit}
\par

\textbf{The proof of Theorem \ref{th1}~:}
We have proved uniform bounds for solution $(u,\psi)$ of \eqref{eq1} in Proposition \ref{Local}. Then we can divide it into two steps to prove Theorem \ref{th1}.

\textbf{Step one. Uniform continuous dependence for $\nu$.}
We are going to show that solution map of \eqref{eq1} is uniform continuous with respect to $\nu\in[0,1]$. If $\nu=0$, since the equation \eqref{eq1} is linear and can be handled more simply, we here omit the details. If $\nu>0$, for any $(u_0,\psi_0),(v_0,\phi_0)\in B_R$, we assume that $(u,\psi),(v,\phi)$ are corresponding solutions of \eqref{eq1}, respectively. Let $\omega=u-v$ and $\Omega=\psi-\phi$. Then we get
\begin{align}\label{eq12}
\left\{
\begin{array}{ll}
\partial_t \omega+u\cdot\nabla\omega-\frac {\epsilon} {Re}\Delta \omega+\nabla(P(u)-P(v))=\frac {1-\epsilon} {Re}div~\tau(\Omega)-\omega\nabla v,  \\[1ex]
\partial_t\Omega+u\cdot\nabla\Omega=div_{R}[-\sigma(u)\cdot{R}\Omega+\frac {1} {2}\psi_{\infty}\nabla_{R}\frac{\Omega}{\psi_{\infty}}-\sigma(\omega)R \phi]-\omega\nabla \phi,  ~~~~~~~\\[1ex]
\omega|_{t=0}=u_0-v_0,~ \Omega|_{t=0}=\psi_0-\phi_0, \\[1ex]
\end{array}
\right.
\end{align}
with $\nabla (P(u)-P(v))=\nabla(\Delta)^{-1}div(\frac {1-\epsilon} {Re}div~\tau(\Omega)-\omega\nabla v-u\cdot\nabla\omega).$ \\
By a similar calculation as in Lemma \ref{estimate2}, we obtain
\begin{align*}
\|\omega\|_{\tilde{L}_{T}^{\infty}(B^{s-1}_{p,r})}&\leq C\big(\|\omega(0)\|_{B^{s-1}_{p,r}}+\int_{0}^{T}
\|\omega\|_{B^{s-1}_{p,r}}(\|u\|_{B^{s}_{p,r}}+\|v\|_{B^{s}_{p,r}})dt+\|\Omega\|_{\tilde{L}^{\infty}_{T}(B^{s-2}_{p,r})(\mathcal{L}^p)}\big)  \\
&\leq C\big(\|\omega(0)\|_{B^{s-1}_{p,r}}+\|\Omega\|_{\tilde{L}^{\infty}_{T}(B^{s-2}_{p,r})(\mathcal{L}^p)}\big).
\end{align*}
Applying Proposition \ref{Local} and Gronwall's inequality, we get
\begin{align*}
\|\omega\|_{\tilde{L}_{T}^{\infty}(B^{s-1}_{p,r})}\leq C\big(\|\omega(0)\|_{B^{s-1}_{p,r}}+\|\Omega\|_{\tilde{L}^{\infty}_{T}(B^{s-2}_{p,r})(\mathcal{L}^p)}\big).
\end{align*}
By Lemma \ref{estimate1}, Lemma \ref{product} with $s_1=s-2,\ s_2=s-1$ and uniform bounds for $u$, $v$, $\psi$ and $\phi$, we have
\begin{align*}
\|\Omega\|_{\tilde{L}_{T}^{\infty}(B^{s-2}_{p,r}(\mathcal{L}^p))}\leq C\big(\|\Omega(0)\|_{B^{s-2}_{p,r}(\mathcal{L}^p)}+T^{\frac 1 2 } \|\omega \|_{\tilde{L}_{T}^{\infty}(B^{s-1}_{p,r})} \big).
\end{align*}
When $T$ is small enough, we can easily deduce that
\begin{align}\label{Ucd1}
\|\omega\|_{\tilde{L}_{T}^{\infty}(B^{s-1}_{p,r})}+\|\Omega\|_{\tilde{L}_{T}^{\infty}(B^{s-2}_{p,r}(\mathcal{L}^p))}\leq C(\|\omega(0)\|_{B^{s-1}_{p,r}}+\|\Omega(0)\|_{B^{s-2}_{p,r}(\mathcal{L}^p)}).
\end{align}

Assume that $(u,\psi),~(u_N,\psi_N)$ are corresponding solutions of \eqref{eq1} with initial data $(u_0,\psi_0),(S_N u_0,S_N \psi_0)$, respectively. Let $\omega_N=u-u_N$ and $\Omega_N=\psi-\psi_N$, then we obtain
\begin{align}\label{eq13}
\left\{
\begin{array}{ll}
\partial_t \omega_N+u\cdot\nabla\omega_N-\frac {\epsilon} {Re}\Delta \omega_N+\nabla(P(u)-P(u_N))=\frac {1-\epsilon} {Re}div~\tau(\Omega_N)-\omega_N\nabla u_N,  \\[1ex]
\partial_t\Omega_N+u\cdot\nabla\Omega_N=div_{R}[-\sigma(u)\cdot{R}\Omega_N+\frac {1} {2}\psi_{\infty}\nabla_{R}\frac{\Omega_N}{\psi_{\infty}}-\sigma(\omega_N)R \psi_N]-\omega_N\nabla \psi_N,  ~~~~~~~\\[1ex]
\omega_N|_{t=0}=u_0-S_N u_0,~ \Omega_N|_{t=0}=\psi_0-S_N\psi_0. \\[1ex]
\end{array}
\right.
\end{align}

Lemma \ref{u}, Lemma \ref{P}  and Lemma \ref{Lemma3} ensure that
\begin{align}\label{Ucd2}
\|\omega_N\|_{\tilde{L}_{T}^{\infty}(B^{s}_{p,r})}&\leq C\big(\|\omega_N(0)\|_{B^{s}_{p,r}}+\int_{0}^{T}
\|\omega_N\|_{B^{s}_{p,r}}(\|u\|_{B^{s}_{p,r}}+\|u_N\|_{B^{s}_{p,r}})+\|\omega_N\nabla u_N\|_{B^{s}_{p,r}}dt+\|\Omega_N\|_{\tilde{L}^{\infty}_{T}(B^{s-1}_{p,r})(\mathcal{L}^p)}\big).
\end{align}
Applying Lemma \ref{product}, we have
\begin{align}\label{Ucd3}
\|\omega_N\nabla u_N\|_{B^{s}_{p,r}}dt&\leq C(\|\omega_N\|_{L^{\infty}}\|\nabla u_N\|_{B^{s}_{p,r}}+\|\omega_N\|_{B^{s}_{p,r}}\|\nabla u_N\|_{L^{\infty}})  \notag \\
&\leq C(\|\omega_N\|_{B^{s-1}_{p,r}}\|u_N\|_{B^{s+1}_{p,r}}+\|\omega_N\|_{B^{s}_{p,r}}\|u_N\|_{B^{s}_{p,r}}).
\end{align}
Plugging \eqref{Ucd3} into \eqref{Ucd2} and applying Gronwall's inequality, we obtain
\begin{align}\label{Ucd4}
\|\omega_N\|_{\tilde{L}_{T}^{\infty}(B^{s}_{p,r})}&\leq C\big(\|\omega_N(0)\|_{B^{s}_{p,r}}+\int_{0}^{T}
\|\omega_N\|_{B^{s-1}_{p,r}}\|u_N\|_{B^{s+1}_{p,r}}dt+\|\Omega_N\|_{\tilde{L}^{\infty}_{T}(B^{s-1}_{p,r})(\mathcal{L}^p)}\big) .
\end{align}
By Proposition \ref{Local} and \eqref{Ucd1}, we get
\begin{align}\label{Ucd5}
\|u_N\|_{\tilde{L}_{T}^{\infty}(B^{s+1}_{p,r})}\leq C(\|u_N(0)\|_{B^{s+1}_{p,r}}+\|\psi_N(0)\|_{B^{s}_{p,r}(\mathcal{L}^p)})\leq C2^N
\end{align}
and
\begin{align}\label{Ucd6}
2^N\|\omega_N\|_{\tilde{L}_{T}^{\infty}(B^{s-1}_{p,r})}\leq C2^N\|\omega_N(0)\|_{B^{s-1}_{p,r}}\leq C\|\omega_N(0)\|_{B^{s}_{p,r}}.
\end{align}
By \eqref{Ucd4}, \eqref{Ucd5} and \eqref{Ucd6}, we have
\begin{align}\label{Ucd7}
\|\omega_N\|_{\tilde{L}_{T}^{\infty}(B^{s}_{p,r})}
\leq C\big(\|\omega_N(0)\|_{B^{s}_{p,r}}+\|\Omega_N\|_{\tilde{L}^{\infty}_{T}(B^{s-1}_{p,r})(\mathcal{L}^p)}\big)  .
\end{align}
By Proposition \ref{Local}, we get
\begin{align}\label{Ucd8}
\|\psi_N\|_{\tilde{L}_{T}^{\infty}(B^{s+1}_{p,r}(\mathcal{L}^p))}\leq C(\|u_N(0)\|_{B^{s+1}_{p,r}}+\|\psi_N(0)\|_{B^{s}_{p,r}(\mathcal{L}^p)})\leq C2^N.
\end{align}
By Lemma \ref{estimate1}, (1) of Lemma \ref{product}, Proposition \ref{Local} and \eqref{Ucd6}, \eqref{Ucd8}, we have
\begin{align}\label{Ucd9}
\|\Omega_N\|_{\tilde{L}_{T}^{\infty}(B^{s-1}_{p,r}(\mathcal{L}^p))}&\leq C\big(\|\Omega_N(0)\|_{B^{s-1}_{p,r}(\mathcal{L}^p)}+T^{\frac 1 2 }( \|\omega_N\nabla \psi_N \|_{\tilde{L}_{T}^{\infty}(B^{s-1}_{p,r})(\mathcal{L}^p)}+  \|\sigma(\omega_N)R \psi_N\|_{\tilde{L}_{T}^{\infty}(B^{s-1}_{p,r})(\mathcal{L}^p)})\big)  \notag \\
&\leq C\big(\|\Omega_N(0)\|_{B^{s-1}_{p,r}(\mathcal{L}^p)}+\|\omega_N(0)\|_{B^{s}_{p,r}}+T^{\frac 1 2 }\|\sigma(\omega_N)R \psi_N\|_{\tilde{L}_{T}^{\infty}(B^{s-1}_{p,r})(\mathcal{L}^p)}\big)  \notag\\
&\leq C\big(\|\Omega_N(0)\|_{B^{s-1}_{p,r}(\mathcal{L}^p)}+\|\omega_N(0)\|_{B^{s}_{p,r}}+T^{\frac 1 2 }\|\omega_N\|_{\tilde{L}_{T}^{\infty}(B^{s}_{p,r})}\big).
\end{align}

Plugging \eqref{Ucd9} into \eqref{Ucd7}, if $T$ is small enough, we can easily deduce that
\begin{align}\label{Ucd10}
\|\omega_N\|_{\tilde{L}_{T}^{\infty}(B^{s}_{p,r})}+\|\Omega_N\|_{\tilde{L}_{T}^{\infty}(B^{s-1}_{p,r}(\mathcal{L}^p))}\leq C(\|\omega_N(0)\|_{B^{s}_{p,r}}+\|\Omega_N(0)\|_{B^{s-1}_{p,r}(\mathcal{L}^p)}).
\end{align}
Denote $X=\tilde{L}_T^{\infty}(B^{s}_{p,r})\times\tilde{L}_T^{\infty}(B^{s-1}_{p,r}(\mathcal{L}^p))$. By Proposition \ref{Local} and \eqref{Ucd1}, \eqref{Ucd10}, we have
\begin{align*}
&\|H^{\nu}(u_0,\psi_0)-H^{\nu}(v_0,\phi_0)\|_{X}\leq \|H^{\nu}(u_0,\psi_0)-H^{\nu}(S_N u_0,S_N \psi_0)\|_{X}  \\
&+\|H^{\nu}(v_0,\phi_0)-H^{\nu}(S_N v_0,S_N \phi_0)\|_{X}
+\|H^{\nu}(S_N u_0,S_N \psi_0)-H^{\nu}(S_N v_0,S_N \phi_0)\|_{X}  \\
&\leq C(\|u_0-S_N u_0\|_{B^{s}_{p,r}}+\|\psi_0-S_N \psi_0\|_{B^{s-1}_{p,r}(\mathcal{L}^p)}+\|v_0-S_N v_0\|_{B^{s}_{p,r}}+\|\phi_0-S_N \phi_0\|_{B^{s-1}_{p,r}(\mathcal{L}^p)}  \\
&+2^{\frac N 2}\|H^{\nu}(S_N u_0,S_N \psi_0)-H^{\nu}(S_N v_0,S_N \phi_0)\|^{\frac 1 2}_{\tilde{L}_T^{\infty}(B^{s-1}_{p,r})\times\tilde{L}_T^{\infty}(B^{s-2}_{p,r}(\mathcal{L}^p))})  \\
&\leq C(\|u_0-S_N u_0\|_{B^{s}_{p,r}}+\|\psi_0-S_N \psi_0\|_{B^{s-1}_{p,r}(\mathcal{L}^p)}+\|v_0-S_N v_0\|_{B^{s}_{p,r}}+\|\phi_0-S_N \phi_0\|_{B^{s-1}_{p,r}(\mathcal{L}^p)}  \\
&+(\|u_0-v_0\|_{B^{s}_{p,r}}+\|\psi_0-\phi_0\|_{B^{s-1}_{p,r}(\mathcal{L}^p)})^{\frac 1 2}).
\end{align*}
Thus, we complete the proof of uniform continuous dependence.

\textbf{Step two. Inviscid limit.}
Assume that $(u^\nu,\psi^\nu),(u^\nu_N,\psi^\nu_N)$ are corresponding solutions of \eqref{eq1} with initial data $(u_0,\psi_0),(S_N u_0,S_N \psi_0)$ for $\nu\in[0,1]$, respectively.  \\
Letting $\omega^\nu_N=u^\nu_N-u^0_N$ and $\Omega^\nu_N=\psi_N^\nu-\psi^0_N$ for $\nu>0$, then we have
\begin{align}\label{eq13}
\left\{
\begin{array}{ll}
\partial_t \omega^\nu_N+u^\nu_N\cdot\nabla\omega^\nu_N+\nabla(P(u^\nu_N)-P(u^0_N))=\frac {\epsilon} {Re}\Delta u^\nu_N+\frac {1-\epsilon} {Re}div~\tau(\psi^\nu_N)-\omega^\nu_N\nabla u^0_N,  \\[1ex]
\partial_t\Omega^\nu_N+u^\nu_N\cdot\nabla\Omega^\nu_N=div_{R}[-\sigma(u^\nu_N)\cdot{R}\Omega^\nu_N+\frac {1} {2}\psi_{\infty}\nabla_{R}\frac{\Omega^\nu_N}{\psi_{\infty}}-\sigma(\omega^\nu_N)R \psi^0_N]-\omega^\nu_N\nabla \psi^0_N,  ~~~~~~~\\[1ex]
\omega^\nu_N|_{t=0}=0,~ \Omega^\nu_N|_{t=0}=0. \\[1ex]
\end{array}
\right.
\end{align}

Lemma \ref{u} , Lemma \ref{P}  and Lemma \ref{Lemma3} ensure that
\begin{align*}
\|\omega^\nu_N\|_{\tilde{L}_{T}^{\infty}(B^{s-1}_{p,r})}\leq C\big(\int_{0}^{T}
\|\omega^\nu_N\|_{B^{s-1}_{p,r}}(\|u^0_N\|_{B^{s}_{p,r}}+\|u^\nu_N\|_{B^{s}_{p,r}})dt+C\nu2^N \big) ,
\end{align*}
which implies
\begin{align}\label{In1}
\|\omega^\nu_N\|_{\tilde{L}_{T}^{\infty}(B^{s-1}_{p,r})}\leq C\nu2^N .
\end{align}
By Lemma \ref{estimate1}, Lemma \ref{product} with $s_1=s-2,\ s_2=s-1$ and uniform bounds for $\psi^0_N$, we have
\begin{align}\label{In2}
\|\Omega^\nu_N\|_{\tilde{L}_{T}^{\infty}(B^{s-2}_{p,r}(\mathcal{L}^p))}&\leq CT^{\frac 1 2 }( \|\omega^\nu_N\nabla \psi^0_N \|_{\tilde{L}_{T}^{\infty}(B^{s-2}_{p,r})(\mathcal{L}^p)}+  \|\sigma(\omega^\nu_N)R \psi^0_N\|_{\tilde{L}_{T}^{\infty}(B^{s-2}_{p,r})(\mathcal{L}^p)})  \notag\\
&\leq C\|\omega^\nu_N\|_{\tilde{L}_{T}^{\infty}(B^{s-1}_{p,r})}\leq C\nu2^N .
\end{align}
Similar, we obtain
\begin{align*}
\|\omega^\nu_N\|_{\tilde{L}_{T}^{\infty}(B^{s}_{p,r})}&\leq C\int_{0}^{T}
\|\omega^\nu_N\|_{B^{s}_{p,r}}(\|u^\nu_N\|_{B^{s}_{p,r}}+\|u^0_N\|_{B^{s}_{p,r}})+\|\omega^\nu_N\nabla u^0_N\|_{B^{s}_{p,r}}dt+C\nu2^{2N} \\
&\leq C\int_{0}^{T}
\|\omega^\nu_N\|_{B^{s}_{p,r}}(\|u^\nu_N\|_{B^{s}_{p,r}}+\|u^0_N\|_{B^{s}_{p,r}})+\|\omega^\nu_N\|_{B^{s-1}_{p,r}}\|u^0_N\|_{B^{s+1}_{p,r}}dt+C\nu2^{2N} .
\end{align*}
Using \eqref{In1}, Proposition \ref{Local} and applying Gronwall's inequality, we deduce that
\begin{align}\label{In3}
\|\omega^\nu_N\|_{\tilde{L}_{T}^{\infty}(B^{s}_{p,r})}\leq C\nu2^{2N}.
\end{align}
By Lemma \ref{estimate1}, Lemma \ref{product}, \eqref{In3} and estimates for $\psi^0_N$, we have
\begin{align}\label{In4}
\|\Omega^\nu_N\|_{\tilde{L}_{T}^{\infty}(B^{s-1}_{p,r}(\mathcal{L}^p))}&\leq CT^{\frac 1 2 }( \|\omega^\nu_N\nabla \psi^0_N \|_{\tilde{L}_{T}^{\infty}(B^{s-1}_{p,r})(\mathcal{L}^p)}+  \|\sigma(\omega^\nu_N)R \psi^0_N\|_{\tilde{L}_{T}^{\infty}(B^{s-1}_{p,r})(\mathcal{L}^p)})  \notag\\
&\leq C(\|\omega^\nu_N\|_{\tilde{L}_{T}^{\infty}(B^{s-1}_{p,r})}\|\psi^0_N \|_{\tilde{L}_{T}^{\infty}(B^{s}_{p,r})(\mathcal{L}^p)}+\|\omega^\nu_N\|_{\tilde{L}_{T}^{\infty}(B^{s}_{p,r})}\|\psi^0_N \|_{\tilde{L}_{T}^{\infty}(B^{s-1}_{p,r})(\mathcal{L}^p)})  \notag \\
&\leq C\nu2^{2N}.
\end{align}

Let $X=\tilde{L}_T^{\infty}(B^{s}_{p,r})\times\tilde{L}_T^{\infty}(B^{s-1}_{p,r}(\mathcal{L}^p))$. By \eqref{Ucd10}, \eqref{In3} and \eqref{In4},  we have
\begin{align*}
&\|H^{\nu}(u_0,\psi_0)-H^{0}(u_0,\psi_0)\|_{X}\leq \|H^{\nu}(u_0,\psi_0)-H^{\nu}(S_N u_0,S_N \psi_0)\|_{X}  \\
&+\|H^{0}(u_0,\psi_0)-H^{0}(S_N u_0,S_N \psi_0)\|_{X}
+\|H^{\nu}(S_N u_0,S_N \psi_0)-H^{0}(S_N u_0,S_N \psi_0)\|_{X}  \\
&\leq C(\|u_0-S_N u_0\|_{B^{s}_{p,r}}+\|\psi_0-S_N \psi_0\|_{B^{s-1}_{p,r}(\mathcal{L}^p)}+\nu2^{2N}).
\end{align*}

Then we deduce that
\begin{align}
\lim_{\nu\rightarrow 0}
\|H^{\nu}(u_0,\psi_0)-H^{0}(u_0,\psi_0)\|_{\tilde{L}_T^{\infty}(B^{s}_{p,r})\times\tilde{L}_T^{\infty}(B^{s-1}_{p,r}(\mathcal{L}^p))}=0.
\end{align}
Combining Proposition \ref{Local}, we complete the proof of Theorem \ref{th1}.     \hfill$\Box$

\section{The rate of convergence in $L^p$}
\par

\textbf{The proof of Theorem \ref{th2}~:}
Assume that $(u^\nu,\psi^\nu)$ is the corresponding solution of \eqref{eq1} with the same initial data $(u_0,\psi_0)$ for $\nu\in[0,1]$.  \\
Letting $\omega^\nu=u^\nu-u^0$ and $\Omega^\nu=\psi^\nu-\psi^0$ for $\nu>0$, then we have
\begin{align}\label{eq14}
\left\{
\begin{array}{ll}
\partial_t \omega^\nu+u^\nu\cdot\nabla\omega^\nu+\nabla(P(u^\nu)-P(u^0))=\frac {\epsilon} {Re}\Delta u^\nu+\frac {1-\epsilon} {Re}div~\tau(\psi^\nu)-\omega^\nu\nabla u^0,  \\[1ex]
\partial_t\Omega^\nu+u^\nu\cdot\nabla\Omega^\nu=div_{R}[-\sigma(u^\nu)\cdot{R}\Omega^\nu+\frac {1} {2}\psi_{\infty}\nabla_{R}\frac{\Omega^\nu}{\psi_{\infty}}-\sigma(\omega^\nu)R \psi^0]-\omega^\nu\nabla \psi^0,  ~~~~~~~\\[1ex]
\omega^\nu|_{t=0}=0,~ \Omega^\nu|_{t=0}=0. \\[1ex]
\end{array}
\right.
\end{align}
Under the assumption of Theorem \ref{th1}, we have $s>1+max\{\frac 1 2, \frac d p\}$, $2\leq p\leq r<\infty$ and $k(p-1)>1$.
Similarly, we get
\begin{align*}
\|\omega^\nu\|_{L_{T}^{\infty}(B^{s-2}_{p,r})}&\leq C\big(\int_{0}^{T}
\|\omega^\nu\|_{B^{s-2}_{p,r}}(\|u^0_N\|_{B^{s}_{p,r}}+\|u^\nu\|_{B^{s}_{p,r}})+\|\omega^\nu\nabla u^0\|_{B^{s-2}_{p,r}}dt+C\nu\big) \\
&\leq C\big(\int_{0}^{T}
\|\omega^\nu\|_{B^{s-2}_{p,r}}dt+C\nu\big).
\end{align*}
Applying Gronwall's inequality, then we obtain
\begin{align}\label{Rate1}
\|\omega^\nu\|_{L_{T}^{\infty}(B^{s-2}_{p,r})}\leq C\nu.
\end{align}
An interpolation argument ensures that
\begin{align}\label{Rate2}
\|\omega^\nu\|_{L_{T}^{\infty}(B^{s-1}_{p,r})}\leq \|\omega^\nu\|^{\frac 1 2}_{L_{T}^{\infty}(B^{s-2}_{p,r})}\|\omega^\nu\|^{\frac 1 2}_{L_{T}^{\infty}(B^{s}_{p,r})}\leq C\nu^{\frac 1 2}.
\end{align}
By Lemma \ref{estimate1}, Lemma \ref{product} and uniform bounds for $\psi^0_N$, we have
\begin{align}\label{Rate3}
\|\Omega^\nu\|_{L_{T}^{\infty}(B^{s-2}_{p,r}(\mathcal{L}^p))}&\leq CT^{\frac 1 2 }( \|\omega^\nu\nabla \psi^0 \|_{L_{T}^{\infty}(B^{s-2}_{p,r})(\mathcal{L}^p)}+  \|\sigma(\omega^\nu_N)R \psi^0\|_{L_{T}^{\infty}(B^{s-2}_{p,r})(\mathcal{L}^p)})  \notag\\
&\leq C\|\omega^\nu_N\|_{L_{T}^{\infty}(B^{s-1}_{p,r})}\leq C\nu^{\frac 1 2} .
\end{align}
Under the assumption $1+max\{\frac 1 2, \frac d p\}<s$, we now prove the rate of convergence in $L^p$.
If $s>2$, we deduce that
$$\|\omega^\nu(t)\|_{L^p}\leq C\|\omega^\nu\|_{L_{T}^{\infty}(B^{s-2}_{p,r})}\leq C\nu,~~\|\Omega^\nu(t)\|_{L_x^p(\mathcal{L}^p)}\leq C\|\Omega^\nu\|_{L_{T}^{\infty}(B^{s-2}_{p,r}(\mathcal{L}^p))}\leq C\nu^{\frac 1 2}.$$
If $s=2$, for any small $\epsilon_1\in(0,1)$, we have
$$\|\omega^\nu\|_{L^p}\leq C\|\omega^\nu\|_{L_{T}^{\infty}(B^{2\epsilon_1}_{p,r})}\leq C\|\omega^\nu\|^{1-\epsilon_1}_{L_{T}^{\infty}(B^{s-2}_{p,r})}\|\omega^\nu\|^{\epsilon_1}_{L_{T}^{\infty}(B^{s}_{p,r})}\leq C\nu^{1-\epsilon_1}.$$
and
$$\|\Omega^\nu\|_{L_x^p(\mathcal{L}^p)}\leq C\|\Omega^\nu\|_{L_{T}^{\infty}(B^{2\epsilon_1}_{p,r}(\mathcal{L}^p))}\leq C\|\Omega^\nu\|^{1-\epsilon_1}_{L_{T}^{\infty}(B^{s-2}_{p,r}(\mathcal{L}^p))}\|\Omega^\nu\|^{\epsilon_1}_{L_{T}^{\infty}(B^{s-1}_{p,r}(\mathcal{L}^p))}\leq C\nu^{\frac {1-\epsilon_1} {2}}.$$
If $s<2$, we get
$$\|\omega^\nu\|_{L^p}\leq \|\omega^\nu\|^{\frac s 2}_{L_{T}^{\infty}(B^{s-2}_{p,r})}\|\omega^\nu\|^{1-\frac s 2}_{L_{T}^{\infty}(B^{s}_{p,r})}\leq C\nu^{\frac s 2}.$$
and
$$\|\Omega^\nu\|_{L_x^p(\mathcal{L}^p)}\leq C\|\Omega^\nu\|^{s-1}_{L_{T}^{\infty}(B^{s-2}_{p,r}(\mathcal{L}^p))}\|\Omega^\nu\|^{2-s}_{L_{T}^{\infty}(B^{s-1}_{p,r}(\mathcal{L}^p))}\leq C\nu^{\frac {s-1} {2}}.$$

We thus complete the proof of Theorem \ref{th2}.     \hfill$\Box$

\smallskip
\noindent\textbf{Acknowledgments} This work was
partially supported by the National Natural Science Foundation of China (No. 11671407 and No. 11701586), the Macao Science and Technology Development Fund (No. 0091/2018/A3), and Guangdong Province of China Special Support Program (No. 8-2015),
and the key project of the Natural Science Foundation of Guangdong province (No. 2016A030311004).


\phantomsection
\addcontentsline{toc}{section}{\refname}
\bibliographystyle{abbrv} 
\bibliography{Feneref}

\begin{thebibliography}{10}

\bibitem{Bahouri2011}
H.~Bahouri, J.-Y. Chemin, and R.~Danchin.
\newblock {\em Fourier analysis and nonlinear partial differential equations},
  volume 343 of {\em Grundlehren der Mathematischen Wissenschaften}.
\newblock Springer, Heidelberg, 2011.

\bibitem{Bernicot}
F.~Bernicot, T.~Elgindi, and S.~Keraani.
\newblock On the inviscid limit of the 2{D} {N}avier-{S}tokes equations with
  vorticity belonging to {BMO}-type spaces.
\newblock {\em Ann. Inst. H. Poincar\'{e} Anal. Non Lin\'{e}aire},
  33(2):597--619, 2016.

\bibitem{Bird1977}
R.~B. Bird, R.~C. Armstrong, and O.~Hassager.
\newblock {\em Dynamics of Polymeric Liquids}, volume~1.
\newblock Wiley, New York, 1977.

\bibitem{Busuioc}
A.~V. Busuioc, I.~S. Ciuperca, D.~Iftimie, and L.~I. Palade.
\newblock The {FENE} dumbbell polymer model: existence and uniqueness of
  solutions for the momentum balance equation.
\newblock {\em J. Dynam. Differential Equations}, 26(2):217--241, 2014.

\bibitem{Doi1988}
M.~Doi and S.~F. Edwards.
\newblock {\em The Theory of Polymer Dynamics}.
\newblock Oxford University Press, Oxford, 1988.

\bibitem{Guo-Li-Yin}
Z.~Guo, J.~Li, and Z.~Yin.
\newblock Local well-posedness of the incompressible {E}uler equations in
  {$B^1_{\infty,1}$} and the inviscid limit of the {N}avier-{S}tokes equations.
\newblock {\em J. Funct. Anal.}, 276(9):2821--2830, 2019.

\bibitem{Jourdain}
B.~Jourdain, T.~Leli\`evre, and C.~Le~Bris.
\newblock Existence of solution for a micro-macro model of polymeric fluid: the
  {FENE} model.
\newblock {\em J. Funct. Anal.}, 209(1):162--193, 2004.

\bibitem{Kato}
T.~Kato.
\newblock Nonstationary flows of viscous and ideal fluids in {${\bf R}^{3}$}.
\newblock {\em J. Functional Analysis}, 9:296--305, 1972.

\bibitem{Li-Yin}
J.~Li and Z.~Yin.
\newblock Remarks on the well-posedness of {C}amassa-{H}olm type equations in
  {B}esov spaces.
\newblock {\em J. Differential Equations}, 261(11):6125--6143, 2016.

\bibitem{F.Lin}
F.~Lin, P.~Zhang, and Z.~Zhang.
\newblock On the global existence of smooth solution to the 2-{D} {FENE}
  dumbbell model.
\newblock {\em Comm. Math. Phys.}, 277(2):531--553, 2008.

\bibitem{Luo-Yin-NA}
W.~Luo and Z.~Yin.
\newblock Global existence and well-posedness for the {FENE} dumbbell model of
  polymeric flows.
\newblock {\em Nonlinear Anal. Real World Appl.}, 37:457--488, 2017.

\bibitem{Luo-Yin}
W.~Luo and Z.~Yin.
\newblock The {L}iouville {T}heorem and the {$L^2$} {D}ecay for the {FENE}
  {D}umbbell {M}odel of {P}olymeric {F}lows.
\newblock {\em Arch. Ration. Mech. Anal.}, 224(1):209--231, 2017.

\bibitem{Luo-Yin2}
W.~Luo and Z.~Yin.
\newblock The {$L^2$} decay for the 2{D} co-rotation {FENE} dumbbell model of
  polymeric flows.
\newblock {\em Adv. Math.}, 343:522--537, 2019.

\bibitem{Majda}
A.~Majda.
\newblock Vorticity and the mathematical theory of incompressible fluid flow.
\newblock volume~39, pages S187--S220. 1986.
\newblock Frontiers of the mathematical sciences: 1985 (New York, 1985).

\bibitem{Masmoudi-2007CMP}
N.~Masmoudi.
\newblock Remarks about the inviscid limit of the {N}avier-{S}tokes system.
\newblock {\em Comm. Math. Phys.}, 270(3):777--788, 2007.

\bibitem{Masmoudi2008}
N.~Masmoudi.
\newblock Well-posedness for the {FENE} dumbbell model of polymeric flows.
\newblock {\em Comm. Pure Appl. Math.}, 61(12):1685--1714, 2008.

\bibitem{Masmoudi2013}
N.~Masmoudi.
\newblock Global existence of weak solutions to the {FENE} dumbbell model of
  polymeric flows.
\newblock {\em Invent. Math.}, 191(2):427--500, 2013.

\bibitem{Renardy}
M.~Renardy.
\newblock An existence theorem for model equations resulting from kinetic
  theories of polymer solutions.
\newblock {\em SIAM J. Math. Anal.}, 22(2):313--327, 1991.

\bibitem{Schonbek1985}
M.~E. Schonbek.
\newblock {$L^2$} decay for weak solutions of the {N}avier-{S}tokes equations.
\newblock {\em Arch. Rational Mech. Anal.}, 88(3):209--222, 1985.

\bibitem{Schonbek}
M.~E. Schonbek.
\newblock Existence and decay of polymeric flows.
\newblock {\em SIAM J. Math. Anal.}, 41(2):564--587, 2009.

\bibitem{Swann}
H.~S.~G. Swann.
\newblock The convergence with vanishing viscosity of nonstationary
  {N}avier-{S}tokes flow to ideal flow in {$R_{3}$}.
\newblock {\em Trans. Amer. Math. Soc.}, 157:373--397, 1971.

\bibitem{Zhang-H}
H.~Zhang and P.~Zhang.
\newblock Local existence for the {FENE}-dumbbell model of polymeric fluids.
\newblock {\em Arch. Ration. Mech. Anal.}, 181(2):373--400, 2006.

\end{thebibliography}

\end{document}